\documentclass[12pt]{article}
\usepackage[margin=1in]{geometry}
\usepackage{amsthm, amsmath,amsfonts,amssymb,euscript,hyperref,graphics,color,slashed,mathrsfs}
\usepackage{graphicx}
\usepackage{comment}
\usepackage{import}
\usepackage{tikz}
\usepackage{latexsym}
\usepackage{mathtools}
\usepackage{appendix}
\usepackage{stmaryrd}

\makeatletter
\newsavebox{\@brx}
\newcommand{\llangle}[1][]{\savebox{\@brx}{\(\m@th{#1\langle}\)}%
  \mathopen{\copy\@brx\kern-0.5\wd\@brx\usebox{\@brx}}}
\newcommand{\rrangle}[1][]{\savebox{\@brx}{\(\m@th{#1\rangle}\)}%
  \mathclose{\copy\@brx\kern-0.5\wd\@brx\usebox{\@brx}}}
\makeatother

\makeatletter
\newcommand{\subjclass}[2][1991]{%
  \let\@oldtitle\@title%
  \gdef\@title{\@oldtitle\footnotetext{#1 \emph{Mathematics subject classification.} #2}}%
}
\newcommand{\keywords}[1]{%
  \let\@@oldtitle\@title%
  \gdef\@title{\@@oldtitle\footnotetext{\emph{Key words and phrases.} #1.}}%
}
\makeatother

\newtheorem{theorem}{Theorem}[section]
\newtheorem*{theorem*}{Theorem}
\newtheorem{lemma}[theorem]{Lemma}
\newtheorem{proposition}[theorem]{Proposition}

\newtheorem{definition}[theorem]{Definition}
\newtheorem{remark}[theorem]{Remark}

\numberwithin{equation}{section}

\begin{document}
\title {Many $p$-adic odd zeta values are irrational}

\author{Li Lai, Johannes Sprang\thanks{The second author was supported by the DFG grant: SFB 1085 “Higher invariants”.}}
\date{}

\maketitle

\begin{abstract}
For any prime $p$ and $\varepsilon>0$ we prove that for any sufficiently large positive odd integer $s$ at least $(c_p-\varepsilon) \sqrt{\frac{s}{\log s}}$ of the $p$-adic zeta values $\zeta_p(3),\zeta_p(5),\dots,\zeta_p(s)$ are irrational. The constant $c_p$ is positive and does only depend on $p$. This result establishes a $p$-adic version of the elimination technique used by Fischler--Sprang--Zudilin and Lai--Yu to prove a similar result  on classical zeta values. The main difficulty consists in proving the non-vanishing of the resulting linear forms. We overcome this problem by using a new irrationality criterion. 
\end{abstract}


\section{Introduction}

A well-known formula of Euler shows that the values of the Riemann zeta function at positive even integers are all non-zero rational multiples of powers of $\pi$. More precisely, Euler has shown for all positive integers $n$ the formula
\[
	\zeta(2n)= -\frac{(2\pi \sqrt{-1})^{2n}}{2(2n)!}B_{2n},
\]
where $B_{2n}$ are the Bernoulli numbers. In particular, all even zeta values are transcendental numbers. This result raises immediately the question about the nature of the odd zeta values $\zeta(2n+1)$. Although we expect that the odd zeta values are transcendental as well, we do not know the transcendence of $\zeta(2n+1)$ for a single value of $n$ at the moment. Nevertheless, a first step in this direction has been obtained by Ap\'ery in \cite{Ap79}, who has established the irrationality of $\zeta(3)$. A major breakthrough was the celebrated theorem of Rivoal and Ball--Rivoal:
\begin{theorem*}[Rivoal, Ball--Rivoal, \cite{Riv00,BR01}]
	For any $\varepsilon>0$ and a sufficiently large odd positive integer $s$
	\[
		\dim_{\mathbb{Q}} \left( \mathbb{Q}+ \zeta(3)\mathbb{Q}+ \zeta(5)\mathbb{Q}+\dots + \zeta(s)\mathbb{Q} \right)\geqslant \frac{1-\varepsilon}{1+\log 2} \log s.
	\]
\end{theorem*}
In particular, this theorem implies that for sufficiently large $s$ there are at least $\frac{1-\varepsilon}{1+\log 2} \log s$ irrational numbers among $\zeta(3),\zeta(5),\dots, \zeta(s)$. Although the theorem of Rivoal and Ball--Rivoal  implies the irrationality of infinitely many odd zeta values,  $\zeta(3)$ remains the only particular zeta value which is known to be irrational. Nevertheless, a beautiful theorem of Zudilin says that at least one of $\zeta(5),\zeta(7), \zeta(9)$ and $\zeta(11)$ is irrational, see \cite{Zud01}.

Recently, there has been further progress on asymptotic results on the irrationality of odd zeta values. It was an important insight of Zudilin that one can use certain linear forms in Hurwitz zeta values to construct linear forms in zeta values with related coefficients, see \cite{Zud18}. Taking suitable linear combinations of such linear forms allows one to eliminate certain unwanted zeta values in the resulting linear forms. The elimination technique turned out to be very useful to improve the lower bound on the number of irrational odd zeta values:
\begin{theorem*}[Fischler--Sprang--Zudilin, \cite{FSZ2019}]
	For any $\varepsilon>0$ and $s$ a sufficiently large positive odd integer,  at least
	\[
		2^{(1-\varepsilon)\frac{\log s}{\log\log s}}
	\]
	of the numbers $\zeta(3),\zeta(5),\dots,\zeta(s)$ are irrational. 
\end{theorem*}
A further improvement of the elimination technique led to the following theorem which gives the best lower bound for the number of irrational odd zeta values at the moment:
\begin{theorem*}[Lai--Yu, \cite{LY2020}]
	For any $\varepsilon>0$ and $s$ a sufficiently large positive odd integer,  at least
	\[
		\left(c_0 -\varepsilon\right)\sqrt{\frac{s}{\log s}}
	\]
	of the numbers $\zeta(3),\zeta(5),\dots,\zeta(s)$ are irrational, where
	\[
		c_0=\sqrt{\frac{4\zeta(2)\zeta(3)}{\zeta(6)}\left( 1-\log \frac{\sqrt{4e^2+1}-1}{2}\right)}\approx 1.192507\dots. 
	\]
\end{theorem*}
The main theorem of this work (Theorem \ref{thmA}) establishes a $p$-adic version of this theorem. Before we turn our attention to $p$-adic zeta values let us mention the following remarkable theorem of Fischler which improves the theorem of Rivoal and Ball--Rivoal considerably:
\begin{theorem*}[Fischler, \cite{Fis21}]
	For any sufficiently large odd positive integer $s$
	\[
		\dim_{\mathbb{Q}} \left( \mathbb{Q}+ \zeta(3)\mathbb{Q}+ \zeta(5)\mathbb{Q}+\dots + \zeta(s)\mathbb{Q} \right)\geqslant 0.21\sqrt{\frac{s}{\log s}}.
	\]
\end{theorem*}

Let us now turn our attention to $p$-adic zeta values. In the following, we will fix a prime $p$ and write as usually $\mathbb{Q}_p$ for the field of $p$-adic numbers. For an integer $s\neq 1$, we define the $p$-adic zeta value $\zeta_p(s):=L_p(s, \omega^{1-s})$, where $\omega$ is the $p$-adic Teichmüller character and $L_p(s,\chi)$ is the Kubota--Leopoldt $p$-adic $L$-function. This definition is justified by the fact that
\[
	\zeta_p(s)=(1-p^{-s})\zeta(s), \quad \text{ for } s\in \mathbb{Z}_{<0}.
\]
On the other hand, it follows from the Kummer congruences that the value for positive integers $s$ can be obtained by a process of $p$-adic interpolation from classical zeta values
\begin{equation}\label{eq:zetap_as_limit}
	\zeta_p(s)=\lim_{\substack{ k\to s \text{ $p$-adically}\\  k\in \mathbb{Z}_{<0},\quad  k\equiv s \mod (p-1) }}\zeta (k), \quad \text{ for } s\in \mathbb{Z}_{> 1}.
\end{equation}
Note that \eqref{eq:zetap_as_limit} implies $\zeta_p(2n)=0$ for $n\in \mathbb{Z}_{>0}$. It turns out that the question about the nature of the values $\zeta_p(2n+1)$ is even more difficult than the corresponding question for classical odd zeta values. For example, it is not even known that $\zeta_p(2n+1)\neq 0$ for all $n$ and all primes $p$. In the following, we will briefly discuss what is known about the irrationality of $p$-adic zeta values. Calegari proved a $p$-adic version of Ap\'ery's theorem for $p=2,3$:
\begin{theorem*}[Calegari, \cite{Cal05}]
$\zeta_p(3)$ is irrational for $p=2,3$.
\end{theorem*}
Beukers gave an alternative proof of the irrationality of $\zeta_p(3)$ for $p=2,3$ and established the irrationality of certain $p$-adic Hurwitz zeta values \cite{Beu08}. Further results on the irrationality of $p$-adic Hurwitz zeta values have been obtained by Bel \cite{Bel10,Bel19}. Very recently, the first author establishes a $2$-adic version of a theorem of Zudilin:
 \begin{theorem*}[Lai, \cite{Lai2023}]
 For any $s \in \mathbb{Z}_{\geqslant 0}$, at least one of the zeta values $\zeta_2(j)$ for $s+3\leqslant j\leqslant 3s+5$ is irrational.
 \end{theorem*}
In particular, this theorem implies that each of the sets $\{\zeta_2(7),\zeta_2(9),\zeta_2(11),\zeta_2(13)\}$ and $\{\zeta_2(5),\zeta_2(7)\}$ contains at least one irrational number. Recently, Calegari, Dimitrov and Tang announced the irrationality of $\zeta_2(5)$, \cite{CDT20}. The following theorem of the second author can be seen as a $p$-adic variant of the theorem of Rivoal and Ball--Rivoal:
\begin{theorem*}[Sprang, \cite{Spr2020}]
Let $p$ be a prime. For any $\varepsilon>0$ and $s$ a sufficiently large positive odd integer, we have
\[
		\dim_{\mathbb{Q}} \left( \mathbb{Q}+ \zeta_p(3)\mathbb{Q}+ \zeta_p(5)\mathbb{Q}+\dots + \zeta_p(s)\mathbb{Q} \right)\geqslant \frac{1-\varepsilon}{2+2\log 2} \log s.
	\]
\end{theorem*}  
Volkenborn integration over rational functions turned out to be a systematic tool to produce rapidly convergent linear forms in $p$-adic Hurwitz zeta values with related coefficients. So it seems to be a natural step to apply the elimination technique to such linear forms. Unfortunately, the arguments used in \cite{FSZ2019,LY2020} in order to show the non-vanishing of the resulting linear forms do not work in the case of $p$-adic zeta values. This is the main reason why the elimination technique has not yet been successfully applied to $p$-adic zeta values. In this paper, we overcome these technical difficulties and establish the following theorem which can be seen as a $p$-adic version of the theorem of Lai--Yu. 
\begin{theorem}\label{thmA}
    Let $p$ be a prime. For any $\varepsilon > 0$ and a sufficiently large positive odd integer $s$, we have
	\[ \#\{ \text{~odd~} j \in [3,s] ~\mid~ \zeta_p(j) \notin \mathbb{Q} \} \geqslant (c_p - \varepsilon)\sqrt{\frac{s}{\log s}},\]
	where the constant
	\[c_p = \sqrt{\frac{4\zeta(2)\zeta(3)}{\zeta(6)}\cdot\frac{\left(l_p+\frac{1}{p-1}\right)\log p - 1 - \log 2}{p^{l_p-2}(p^2-p+1)}},  \]
	and 
	\[ l_p = \begin{cases}
		1, &\text{if~} p \geqslant 5, \\
		2, &\textit{if~} p=3, \\
		3, &\textit{if~} p=2.
	\end{cases} \]
\end{theorem} 

\begin{remark}
Note that for a given prime $p$, the value for $l_p$ in Theorem \ref{thmA} is chosen in such a way that it maximizes the function
\[
 \mathbb{N}\to \mathbb{R}, \quad l_p\mapsto \frac{\left(l_p+\frac{1}{p-1}\right)\log p - 1 - \log 2}{p^{l_p-2}(p^2-p+1)}.
\]
\end{remark}

One of the key ingredients in the proof of Theorem \ref{thmA} is a new but elementary irrationality criterion (Lemma \ref{lem:irrationalityCrit}). Like other irrationality criteria, this criterion requires a sequence of rapidly converging linear forms $(L_n)_n$ with integer coefficients. The novel aspect of our criterion is that we replace the usual non-vanishing condition on the linear forms by a certain condition on the $\ell(n)$-adic valuation of the coefficients of the forms $L_n$, where $\ell(n)$ is an increasing sequence of auxiliary primes.

\bigskip

Let us briefly outline the structure of the paper: In section \ref{sec:Prelim}, we state and prove our new irrationality criterion. We also discuss basic facts about $p$-adic integration and $p$-adic Hurwitz zeta functions. Furthermore, we relate the Volkenborn integral of primitives of certain functions to the $p$-adic Bernoulli measure. This discussion will be used later to prove the $p$-adic convergence of our linear forms. In section \ref{sec:Rational}, we define a sequence of rational functions. These rational functions are defined similarly as the rational functions used in  \cite{LY2020}. The main difference is the term corresponding to the zero $\theta_{\mathrm{max}}$ in its numerator. This factor will play an important role when we verify the $\ell(n)$-adic properties of the resulting linear forms. In section \ref{sec:LinearForms}, we use Volkenborn integration over primitives of these rational functions to construct many linear forms in $p$-adic zeta values with related coefficients. In section \ref{sec:Arithmetic}, we study the arithmetic properties of the resulting linear forms. Furthermore, we check that the condition on the $\ell(n)$-adic valuation appearing in our irrationality criterion is fulfilled. In section \ref{sec:p-adic}, we use the relation between Volkenborn integration and  the $p$-adic Bernoulli measure to show the $p$-adic convergence of the linear forms. The Archimedean norm of the coefficients of the linear forms is bounded in section \ref{sec:Archimedean}. Finally, the proof of the main theorem is performed in section \ref{sec:MainThm}. The proof is very close to the elimination argument in \cite{FSZ2019,LY2020}. The main difference is that we replace the usual irrationality criterion by  Lemma \ref{lem:irrationalityCrit}.

\section*{Acknowledgement}
We would like to thank the referee for the careful reading of our manuscript, as well as for their invaluable comments and insights, which significantly enhanced the quality of the paper.

\section{Preliminaries}\label{sec:Prelim}

In this section, we formulate and prove a new but elementary irrationality criterion. Furthermore, we recall basic facts about Volkenborn integrals and $p$-adic Hurwitz zeta functions. In the following, $p$ will always denote a prime number and we will write $v_p\colon \mathbb{Q}_p\rightarrow \mathbb{Z}$ for the $p$-adic valuation with the normalization $v_p(p)=1$ and $|x|_p:=p^{-v_p(x)}$ for the $p$-adic norm on $\mathbb{Q}_p$.

\subsection{An irrationality criterion}
In the following, we will prove a new variant of the following elementary $p$-adic irrationality criterion:

\begin{lemma}\label{lem:irrationalityCrit_classical}
Let $\xi_0,\dots,\xi_s\in \mathbb{Q}_p$ and 
\[
	L_n:=l_{0,n}X_0+\dots + l_{s,n}X_s\in \mathbb{Z}[X_0,\dots, X_s] 
\]
be a sequence of linear forms with integral coefficients. Assume that
\[
    0<\max_{0\leqslant i\leqslant s}|l_{i,n}|\cdot |L_n(\xi_0,\dots,\xi_s)|_p\rightarrow 0, \quad \text{ as }n\to \infty.
\] 
Then at least one of $\xi_0,\dots,\xi_s$ is irrational. 
\end{lemma}
\begin{proof}
    We refer to \cite[Lemma 2.1]{Lai2023} for a proof.
\end{proof}

In order to apply Lemma \ref{lem:irrationalityCrit_classical} one has to show the non-vanishing of $L_n(\xi_0,\dots,\xi_s)$. This is often a non-trivial task. The following variant of Lemma \ref{lem:irrationalityCrit_classical} replaces the non-vanishing condition by a condition on the $\ell(n)$-adic valuation of the coefficients of $L_n$, where $\ell(n)$  will be a sequence of auxiliary primes.

\begin{lemma}\label{lem:irrationalityCrit}
Let $\xi_0,\dots,\xi_s\in \mathbb{Q}_p$ and assume that $\xi_0\neq 0$. Let 
\[
	L_n:=l_{0,n}X_0+\dots + l_{s,n}X_s\in \mathbb{Z}[X_0,\dots, X_s] 
\]
be a sequence of non-trivial linear forms. Assume that there is an unbounded subset $I\subseteq\mathbb{N}$ and for each $n\in I$ a prime $\ell(n)$ such that the following three conditions hold:
\begin{enumerate}
\item\label{lem:irrationalityCrit:i} $ \max_{0\leqslant i\leqslant s}|l_{i,n}|\cdot |L_n(\xi_0,\dots,\xi_s)|_p\rightarrow 0$ as $n\to \infty$,
\item\label{lem:irrationalityCrit:ii} For all $n\in I$, we have $v_{\ell(n)}(l_{0,n})<v_{\ell(n)}(l_{i,n})$ for $i=1,\dots,s$.
\item\label{lem:irrationalityCrit:iii} We have $\ell(n)\to \infty $ as $n\to \infty$.
\end{enumerate}
Then at least one of $\xi_0,\dots,\xi_s$ is irrational. 
\end{lemma}
\begin{proof}
If one of $\xi_0\dots, \xi_s$ is irrational, we are done. Otherwise $L_n(\xi_0,\dots,\xi_s)\in \mathbb{Q}$, so it makes sense to take the $\ell(n)$-adic valuation. Since $\ell(n)\to \infty$ as $n\to \infty$, we have $v_{\ell(n)}(\xi_0)=0$ and $v_{\ell(n)}(\xi_i)\geqslant 0$ for all $i=1,\dots, s$ and almost all $n\in I$. Hence, we have
\[
	v_{\ell(n)}(L_n(\xi_0,\dots,\xi_s))= v_{\ell(n)}(l_{0,n}\cdot \xi_0)=v_{\ell(n)}(l_{0,n})<\infty ,
\]
for almost all $n\in I$. In particular, $L_n(\xi_0,\dots,\xi_s)$ is non-zero for almost all $n \in I$. Now, we can apply Lemma \ref{lem:irrationalityCrit_classical} along the subsequence of non-zero linear forms to obtain a contradiction to the rationality of $\xi_0,\dots, \xi_s$.
\end{proof}

Of course, the above lemma also has an Archimedean analogue which might be useful for the investigation of classical zeta values:
\begin{lemma}\label{lem:irrationalityCrit_R}
Let $\xi_0,\dots,\xi_s\in \mathbb{R}$ and assume that $\xi_0\neq 0$. Let 
\[
	L_n:=l_{0,n}X_0+\dots + l_{s,n}X_s\in \mathbb{Z}[X_0,\dots, X_s] 
\]
be a sequence of non-trivial linear forms. Assume that there is an unbounded subset $I\subseteq\mathbb{N}$ and for each $n\in I$ a prime $\ell(n)$ such that the following three conditions hold:
\begin{enumerate}
\item\label{lem:irrationalityCrit_R:i} $|L_n(\xi_0,\dots,\xi_s)|\rightarrow 0$ as $n\to \infty$,
\item\label{lem:irrationalityCrit_R:ii} For all $n\in I$, we have $v_{\ell(n)}(l_{0,n})<v_{\ell(n)}(l_{i,n})$ for $i=1,\dots,s$.
\item\label{lem:irrationalityCrit_R:iii} We have $\ell(n)\to \infty $ as $n\to \infty$.
\end{enumerate}
Then at least one of $\xi_0,\dots,\xi_s$ is irrational. 
\end{lemma}
\begin{proof}
    The proof is essentially the same as in the $p$-adic case, except that Lemma \ref{lem:irrationalityCrit_classical} is replaced by its Archimedean analogue.
\end{proof}

\subsection{$p$-adic integration}
In the following, we recall basic facts about Volkenborn integrals. For more details, we refer the reader to \cite[Chapter 5]{Rob00} or \cite[\S 55]{Sch06}. A continuous function $f\colon \mathbb{Z}_p\to \mathbb{Q}_p$ is called \emph{Volkenborn integrable} if the sequence
\[
\frac{1}{p^n}\sum_{0\leqslant k<p^n} f(k)
\]
converges $p$-adically. In this case, we define
\[
    \int_{\mathbb{Z}_p}f(t) \mathrm{d}t:= \lim_{n \to \infty} \frac{1}{p^n}\sum_{0\leqslant k<p^n} f(k).
\]
Important examples of Volkenborn integrable functions are given by strictly differentiable functions. The Volkenborn integral is not translation invariant, indeed it has the following behaviour under translations:
\begin{lemma}\label{lem:Volkenborn_translation}
    Let $f\colon\mathbb{Z}_p\to \mathbb{Q}_p$ be a strictly differentiable function and $m$ a positive integer, then
    \[
    \int_{\mathbb{Z}_p} f(t+m)\mathrm{d}t= \int_{\mathbb{Z}_p} f(t)\mathrm{d}t+\sum_{i=0}^{m-1}f'(i).
    \]
\end{lemma}
\begin{proof}
    See \cite[\S 5.3, Prop. 2]{Rob00}.
\end{proof}
In the next section, we will discuss the relation between Volkenborn integrals and $p$-adic zeta values in more detail. For the moment, let us only observe that the Bernoulli numbers appear as the moments of the Volkenborn integral:
\begin{lemma}\label{lem:Volkenborn_Bernoulli}
Let $n$ be a non-negative integer and $x\in \mathbb{Q}_p$. Then
\[
    \int_{\mathbb{Z}_p} (x+t)^n \mathrm{d}t= \mathbb{B}_n(x),
\]
where $\mathbb{B}_n$ is the $n$-th Bernoulli polynomial defined by the generating series
\[
    \frac{te^{X\cdot t}}{e^t-1}=\sum_{n=0}^\infty \mathbb{B}_n(X)\frac{t^n}{n!}.
\]
In particular, for $x=0$ the above integral gives the $n$-th Bernoulli number $B_n=\mathbb{B}_n(0)$.
\end{lemma}
\begin{proof}
See \cite[Ch. 5, \S 5.4]{Rob00}.
\end{proof}

Our next goal is to define a functional on overconvergent functions which computes the Volkenborn integral of primitives. For $\rho>0$, let us consider the $\mathbb{Q}_p$-algebra of $\mathbb{Q}_p$-analytic functions on $\mathbb{Z}_p$ of radius of convergence $\geqslant\rho$ 
\[
	C^{\text{an}}_\rho(\mathbb{Z}_p,\mathbb{Q}_p):=\left\{ f\in C(\mathbb{Z}_p,\mathbb{Q}_p) \mid f(x)=\sum_{k=0}^\infty a_kx^k \text{ with } |a_k|_p\rho^k\to 0 \text{ as }k\to \infty \right\}.
\]
The $\mathbb{Q}_p$-algebra $C^{\text{an}}_\rho(\mathbb{Z}_p,\mathbb{Q}_p)$ is a $\mathbb{Q}_p$-Banach algebra equipped with the  norm (see \cite[\S 6.1.5]{BGR84})
\[
	\left|\sum_{k=0}^\infty a_kt^k\right|_\rho:=\max_{k \geqslant 0} |a_k|_p\rho^k.
\]
In the following, we view $C^{\text{an}}_\rho(\mathbb{Z}_p,\mathbb{Q}_p)$ as a topological ring with the topology induced by $|\cdot|_\rho$. 
The $\mathbb{Q}_p$-algebra of \emph{overconvergent functions} on $\mathbb{Z}_p$ is given by
\[
	C^{\dagger}(\mathbb{Z}_p,\mathbb{Q}_p):=\varinjlim_{\rho>1} C^{\text{an}}_\rho(\mathbb{Z}_p,\mathbb{Q}_p).
\]
So $C^{\dagger}(\mathbb{Z}_p,\mathbb{Q}_p)$ consists of all $\mathbb{Q}_p$-analytic functions on $\mathbb{Z}_p$ of radius of convergence strictly larger than $1$.

\begin{definition}
Let $n$ be a positive integer. The continuous $\mathbb{Q}_p$-linear functional on $C^{\dagger}(\mathbb{Z}_p,\mathbb{Q}_p)$ defined by
\[
	\mathcal{L}_n\colon C^{\dagger}(\mathbb{Z}_p,\mathbb{Q}_p)\to \mathbb{Q}_p,\quad f=\sum_{k=0}^\infty a_kt^k\mapsto \mathcal{L}_n(f) :=\sum_{k=0}^\infty n\cdot a_k \frac{B_{k+n}}{k+n}
\]
is called the \emph{$n$-th Bernoulli functional} on $C^{\dagger}(\mathbb{Z}_p,\mathbb{Q}_p)$.
Note that the overconvergence of $f$ together with the von Staudt-Clausen congruence on Bernoulli numbers implies the convergence of the sum on the right-hand side.
\end{definition}
The following lemma relates the Volkenborn integral to the first Bernoulli functional. We will use this to estimate the $p$-adic norm of the Volkenborn integral of primitives of certain rational functions. Note that the derivative of an overconvergent function is again overconvergent.
\begin{lemma}\label{lem:comparison_Volkenborn_Bernoulli}
    For any $f\in C^{\dagger}(\mathbb{Z}_p,\mathbb{Q}_p)$ with $f(0)=0$ we have the formula
    \begin{equation}\label{eq:comparison_Volkenborn_Bernoulli}
	   \mathcal{L}_1(f')=\int_{\mathbb{Z}_p} f(t) \mathrm{d}t.
    \end{equation}
\end{lemma}
\begin{proof}
This follows immediately from the following computation: For $f(t)=\sum_{k\geqslant 1} a_kt^k \in C^{\dagger}(\mathbb{Z}_p,\mathbb{Q}_p)$, we have
\[
\mathcal{L}_1(f')	= \mathcal{L}_1\left( \sum_{k\geqslant 0} a_{k+1}\cdot(k+1) t^k \right)=\sum_{k\geqslant 1}  a_{k} B_{k}=\int_{\mathbb{Z}_p} f(t) \mathrm{d}t.
\]
\end{proof}

Let us finally relate the $n$-th Bernoulli functional to $p$-adic Bernoulli measures. Although the comparison will not be important for the rest of the paper, it compares the $n$-th Bernoulli functional to $p$-adic measures which play an important role in the construction and study of $p$-adic $L$-functions. Let us first recall the definition of $p$-adic measures and $p$-adic distributions:
A $p$-adic distribution on $\mathbb{Z}_p$ is a map $\mu\colon \mathcal{U}\rightarrow \mathbb{Q}_p$ on the set $\mathcal{U}$ of all compact-open subsets of $\mathbb{Z}_p$ which is additive, i.e. it satisfies
\[
	\mu\left(\bigcup_{i=0}^n U_i\right )=\sum_{i=0}^n \mu(U_i)
\]
for any finite collection $U_0,\dots,U_n\in \mathcal{U}$ of pairwise disjoint compact-open subsets. A $p$-adic distribution $\mu$ is called a \emph{$p$-adic measure on $\mathbb{Z}_p$} if and only if it is bounded, i.e.
\[
	\sup_{U\in \mathcal{U}} |\mu(U)|_p<\infty.
\]
Since any compact-open subset of $\mathbb{Z}_p$ can be written as a disjoint union of compact-open sets of the form $b+p^N\mathbb{Z}_p$ with $b,N\in \mathbb{Z}$ and $N\geqslant 0$, a $p$-adic distribution is uniquely determined by its values on such sets. Of course, any $p$-adic distribution can be integrated over locally constant functions with values in $\mathbb{C}_p$. The boundedness assumption in the definition of a $p$-adic measure allows us to define the integral with respect to $\mu$ for all continuous functions $f\colon \mathbb{Z}_p\rightarrow \mathbb{C}_p$ using $p$-adic Riemann sums:
\[
	\mu(f):=\int_{\mathbb{Z}_p} f \mathrm{d}\mu:=\lim_{N\to \infty}  \sum_{b=0}^{p^N-1} f(x_{b,N}) \mu(b+p^N\mathbb{Z}_p),
\]
where $x_{b,N}$ is an arbitrary element of $b+p^N\mathbb{Z}_p$. It is not difficult to see that the limit does not depend on the chosen representatives. Often it is convenient to choose $x_{b,N}=b$, i.e.
\begin{equation}\label{eq:Riemann_sum}
	\mu(f):=\int_{\mathbb{Z}_p} f \mathrm{d}\mu:=\lim_{N\to \infty}  \sum_{b=0}^{p^N-1} f(b) \mu(b+p^N\mathbb{Z}_p).
\end{equation}
In the following, we will prefer to write $\mu(f)$ instead of $\int_{\mathbb{Z}_p} f \mathrm{d}\mu$. In particular, this notation will be useful to avoid any confusion with the Volkenborn integral. 
Let us recall the definition of the $n$-th Bernoulli distribution, see e.g. \cite[Ch. II, \S 5]{Kob84} and \cite[\S 12.2]{Was97} for more details:
\begin{definition}
Let $n\geqslant 0$ be a non-negative integer. The \emph{$n$-th Bernoulli distribution} $\mu_n$ is the unique $p$-adic distribution such that
\[
	\mu_n(b+p^N \mathbb{Z}_p)	:=p^{N(n-1)}\mathbb{B}_n\left( \left\{ \frac{b}{p^N} \right\}\right)
\]
for all $b\in \mathbb{Z}$ and $N\geqslant 0$. Here, $\{x \}:=x-\lfloor x\rfloor$ is the fractional part of a rational number and $\mathbb{B}_n(X):=\sum_{k=0}^n \binom{n}{k}B_{n-k} X^k$ denotes the $n$-th Bernoulli polynomial.
\end{definition}
Note that the $0$-th Bernoulli distribution is the unique normalized translation invariant distribution, therefore it is also sometimes called $p$-adic  Haar distribution. Nevertheless, one can still use formula \eqref{eq:Riemann_sum} to define a notion of $p$-adic integration for certain continuous functions. This naturally leads to the notion of the Volkenborn integral, we have studied before. It is not difficult to see that none of the Bernoulli distributions is bounded. However, this can be fixed by the following regularization process. Let $\alpha\in \mathbb{Z}_p^\times$ and $\mu$ be a $p$-adic distribution. Multiplication by $\alpha$ defines a homeomorphism of $\mathbb{Z}_p$ and we can define a $p$-adic distribution $\alpha_*\mu$ by $\alpha_*\mu (U):=\mu(\alpha^{-1}\cdot U)$. It can be shown that for any non-negative integer $n$ and any $\alpha\in \mathbb{Z}_p^\times$ the $p$-adic distribution
\[
	\mu_{n,\alpha}:=\mu_n-\alpha^n\cdot \alpha_*\mu_{n}
\]
is bounded and hence a $p$-adic measure. The $p$-adic measure $\mu_{n,\alpha}$ is called the \emph{$n$-th regularized Bernoulli measure}. These measures play an important role in Iwasawa theory. Our next goal is to relate the  Bernoulli distributions to the Bernoulli functionals. 
\begin{proposition}\label{prop:Comparison_Bernoulli}
	Let $n$ be a non-negative integer and $\alpha\in \mathbb{Z}_p^\times$ a $p$-adic unit. Then we have the  equality 
\[
	\mu_{n,\alpha}=\mathcal{L}_n -\alpha^n\cdot \alpha_*\mathcal{L}_{n}
\]	
	of continuous $\mathbb{Q}_p$-linear functionals
	\[
		C^{\dagger}(\mathbb{Z}_p,\mathbb{Q}_p)\to \mathbb{Q}_p.
	\]
	Here, we view $\mu_{n,\alpha}=\mu_n-\alpha^n\cdot \alpha_*\mu_{n}$ as a continuous $\mathbb{Q}_p$-linear functional by integration, and $(\alpha_*\mathcal{L}_{n})(f(t)):=\mathcal{L}_n(f(\alpha\cdot t))$.
\end{proposition}
\begin{proof}
It is enough to compare the values of both functionals on the overconvergent functions $t\mapsto t^k$. By the definition of the Bernoulli functional, we have
\[
	(\mathcal{L}_n -\alpha^n\cdot \alpha_*\mathcal{L}_{n})(t^k)=n\cdot (1-\alpha^{k+n})\frac{B_{k+n}}{k+n}.
\]
On the other hand, the following formula for the `moments' of the regularized Bernoulli measure $\mu_{n,\alpha}$ is well known:
\begin{equation}\label{eq:Bernoulli_moments}
	\mu_{n,\alpha}(t^k)=n\cdot (1-\alpha^{k+n})\frac{B_{k+n}}{k+n}.
\end{equation}
For the convenience of the reader, let us sketch the proof of \eqref{eq:Bernoulli_moments}:
\begin{align*}
	\mu_{n,\alpha}(t^k) &= \lim_{N\to \infty} \sum_{b=0}^{p^N-1} b^k \mu_{n,\alpha}(b+p^N\mathbb{Z}_p)\\
	&\stackrel{(\star)}{=} n \cdot \lim_{N\to \infty} \sum_{b=0}^{p^N-1} b^{k+n-1} \mu_{1,\alpha}(b+p^N\mathbb{Z}_p)\\
	&\stackrel{(\star)}{=} \frac{n}{k+n} \cdot \lim_{N\to \infty} \sum_{b=0}^{p^N-1} \mu_{k+n,\alpha}(b+p^N\mathbb{Z}_p)\\
	&= \frac{n}{k+n} \cdot \mu_{k+n,\alpha}(\mathbb{Z}_p)=\frac{n}{k+n} (1-\alpha^{k+n})\cdot \mathbb{B}_{k+n}(0)\\
	&= n\cdot (1-\alpha^{k+n})\frac{B_{k+n}}{k+n}.
\end{align*}
Here, we have used \cite[II, Thm. 5]{Kob84} in the equations labelled with $(\star)$.
\end{proof}

Proposition \ref{prop:Comparison_Bernoulli} allows us to think about the $n$-th Bernoulli functional as a de-regularization of the $n$-th regularized $p$-adic Bernoulli measure on the space of overconvergent power series.

\subsection{$p$-adic Hurwitz zeta functions}
In this section, we recall basic facts about $p$-adic Hurwitz zeta functions and related functions.
It is convenient to define $q_p:=p$ if $p$ is an odd prime, and $q_2:=4$. The units $\mathbb{Z}_p^\times$ of the $p$-adic integers decompose canonically
\[
	\mathbb{Z}_p^\times \xrightarrow{\sim} \mu_{\varphi(q_p)}(\mathbb{Z}_p)\times (1+q_p\mathbb{Z}_p).
\]
Here, $\mu_n(R)$ denotes the group of $n$-th roots of unity in a ring $R$ and $\varphi(n)$ denotes Euler's totient function. The canonical projection
\[
	\omega \colon \mathbb{Z}_p^\times \rightarrow \mu_{\varphi(q_p)}(\mathbb{Z}_p)
\]
is called the \emph{Teichm\"uller character}. Let us extend the Teichm\"uller character to a map
\[
	\mathbb{Q}_p^\times\rightarrow \mathbb{Q}_p^\times,
\]
 by setting
\[
	\omega(x):=p^{v_p(x)}\omega(x/p^{v_p(x)}),
\]
and define $\langle x\rangle:=\frac{x}{\omega(x)}$ for $x\in\mathbb{Q}_p^\times$. For $x\in\mathbb{Q}_p$ with $|x|_p\geqslant q_p$, there is a unique $p$-adic meromorphic function $\zeta_p(s,x)$ on
\[
	\{s\in\mathbb{C}_p\setminus\{1\} \mid |s|_p<q_p p^{-1/(p-1)}\}
\]
such that 
\[
	\zeta_p(1-n,x)=-\omega(x)^{-n}\frac{\mathbb{B}_n(x)}{n}, \quad (n\geqslant 2).
\]
Using Lemma \ref{lem:Volkenborn_Bernoulli}, it is not difficult to show the existence of such a function using Volkenborn integrals. Indeed, we can define $\zeta_p(s,x)$ explicitly as follows:
\begin{equation}\label{eq:zeta_Volkenborn}
	\zeta_p(s,x)=\frac{1}{s-1}\int_{\mathbb{Z}_p}\langle x+t\rangle^{1-s}\mathrm{d}t,
\end{equation}
see \cite[Def. 11.2.5]{Coh07} for more details. Finally, we will need the following lemma which relates the $p$-adic Hurwitz zeta function to $p$-adic zeta values:
\begin{lemma}\label{lem:padicHurwitz_zeta}
    Let $D$ and $i$ be positive integers with $q_p|D$ and $i\geqslant 2$, then
    \[
        D^i\cdot \zeta_p(i)=\sum_{\substack{1\leqslant j\leqslant D \\ \gcd(j,p)=1 }} \omega\left(\frac{j}{D}\right)^{1-i}\zeta_p\left(i,\frac{j}{D}\right).
    \]
\end{lemma}
\begin{proof}
    This follows from $\zeta_p(i)=L_p(i,\omega^{1-i})$ and the more general formula for the Kubota--Leopoldt $p$-adic $L$-function of a Dirichlet character $\chi$ of conductor $f$
    \[
    L_p(i,\chi)=\frac{\langle D \rangle^{1-i}}{D} \sum_{\substack{1\leqslant j\leqslant D \\ \gcd(j,p)=1 }} \chi(j)\zeta_p\left(i,\frac{j}{D}\right),
    \]
    where $D$ is an arbitrary common multiple of $f$ and $q_p$, see \cite[Prop. 11.3.8.(1)]{Coh07}.
\end{proof}

\section{Rational functions}\label{sec:Rational}
The goal of this section is to define sequences of rational functions which will serve as the main input for the construction of linear forms in $p$-adic Hurwitz zeta values.

Fix any prime number $p$. Let $s$ be a positive odd integer and $B$ be a positive real number. We always assume that $s$ and $B$ are larger than some constant depending at most on $p$. Eventually we will take $B = \widetilde{c}_p \sqrt{s/\log s}$ for some constant $\widetilde{c}_p>0$ depending only on $p$ and $\varepsilon$. We define the integer $l_p$ as in the statement of Theorem \ref{thmA}.

\begin{definition}
We define the following two sets depending on $B$:
    \begin{align*}
    \Psi_B &:=   \left\{ b \in \mathbb{N} ~\mid~ \varphi(p^{l_p}b) \leqslant B \right\},\\
    \mathcal{Z}_B &:= \left\{ \frac{a}{p^{l_p}b} ~\mid~ b \in \Psi_B,~1 \leqslant a \leqslant p^{l_p}b,~\gcd(a,p^{l_p}b)=1   \right\}.
    \end{align*}
\end{definition}

\begin{lemma}\label{lemma_inverse_totient}
	Let $N$ be a positive integer, then we have
	\[ \lim_{x \rightarrow +\infty} \frac{1}{x} \cdot \#\left\{ b \in \mathbb{N} ~\mid~ \varphi(Nb) \leqslant x \right\} = \frac{\zeta(2)\zeta(3)}{\zeta(6)}\cdot \frac{1}{N}\prod_{q \mid N} \frac{q^2}{q^2-q+1}, \]
 where the product runs over all prime divisors of $N$.
\end{lemma}

\begin{proof}
	It follows from the Wiener-Ikehara theorem. We only need to slightly modify the arguments in \cite[\S 2]{Bat1972}.
\end{proof}

\begin{lemma}\label{lemma_sizes_of_sets}
    As $B \rightarrow +\infty$, we have the following asymptotical estimates for the sizes of the sets $\Psi_B$ and $\mathcal{Z}_B$:  
	\begin{align}
            |\Psi_B| &= (a_p + o(1))B, \label{size_of_Psi_B}\\
            |\mathcal{Z}_B| &= \left( \frac{a_p}{2} + o(1)\right)B^2, \label{size_of_Z_B}
	\end{align} 
    where
    \begin{equation}\label{def_a_p}
        a_p := \frac{\zeta(2)\zeta(3)}{\zeta(6)}\cdot\frac{1}{p^{l_p-2}(p^2-p+1)}.
    \end{equation}
	
\end{lemma}

\begin{proof}
	Consider the function 
	\[ f(x) := \#\left\{ b \in \mathbb{N} ~\mid~ \varphi(p^{l_p}b) \leqslant x \right\}, \quad x \in (0,+\infty). \]
By Lemma \ref{lemma_inverse_totient}, we have $f(x) \sim a_p x$ as $x \to +\infty$. The estimate \eqref{size_of_Psi_B} follows immediately since $|\Psi_B|=f(B)$.

For \eqref{size_of_Z_B}, we first express $|\mathcal{Z}_B|$ as the following Riemann-Stieltjes integral:
\[ |\mathcal{Z}_B| = \sum_{b \in \Psi_B} \varphi\left(p^{l_p}b\right) =\int_{1^{-}}^{B} x\mathrm{d}f(x). \]
By integration by parts, we have
\[ |\mathcal{Z}_B| = Bf(B) - \int_{1^{-}}^{B} f(x) \mathrm{d}x.  \]
Since $f(x) \sim a_p x$ as $x \to +\infty$, we obtain that $|\mathcal{Z}_B| \sim a_pB^2/2$ as $B \to +\infty$. The proof of Lemma \ref{lemma_sizes_of_sets} is complete.
\end{proof}

\bigskip

Let us define $b_{\max}:=\max \Psi_B$, $\theta_{\max}:=\max \mathcal{Z}_B=\frac{p^{l_p}b_{\max}-1}{p^{l_p}b_{\max}}$ and $\theta_{\min}:=\min \mathcal{Z}_B=1-\theta_{\max} = \frac{1}{p^{l_p}b_{\max}}$. Define the integer
\[ P_{B} := \operatorname{LCM}\left\{q-1 ~\mid~ q \text{~is a prime divisor of~} p^{l_p}b \text{~for some~} b \in \Psi_B \right\}. \]
($\operatorname{LCM}$ means taking the least common multiple.) Define the set 
\[I:=\{n \in \mathbb{N} ~\mid~ n+1 \text{~is a prime, and~} n \text{~is a multiple of~} P_{B}p^{l_p}b_{\max}\}. \]
For $n\in I$, we define the auxiliary prime
\[\ell(n):=n+1.\] 
Note that $I$ is an unbounded subset of $\mathbb{N}$ by Dirichlet's theorem on primes in arithmetic progressions. The sequence $(\ell(n))_{n\in I}$ will be the sequence of primes which is needed in the irrationality criterion Lemma \ref{lem:irrationalityCrit}.

For a positive integer $k$, the Pochhammer symbol $(\alpha)_k$ is defined by $(\alpha)_k := \alpha(\alpha+1)\cdots(\alpha+k-1)$.

\begin{definition}\label{definition_Rn(t)}
    Fix an odd integer $s$ and a positive real number $B$ such that 
    \[s > |\mathcal{Z}_B| > 2.\] 
    We define for $n\in I$ the rational function $R_n(t)\in \mathbb{Q}(t)$ by
        \begin{align*}
            R_n(t) :=  A_1(B)^n \cdot A_2(B)^n \cdot \frac{n!^{s-|\mathcal{Z}_B|+1}}{(\theta_{\max}n-1)!} \cdot \frac{ \left(t+\theta_{\max} \right)_{\theta_{\max}n-1}\prod_{\theta \in \mathcal{Z}_B\setminus\{\theta_{\max}\}}(t+\theta)_{n}}{(t)_{n+1}^{s}},
        \end{align*}
    where
	\begin{align*}
            A_{1}(B) &:= \prod_{b \in \Psi_B} \left( p^{l_p}b \right)^{\varphi\left(p^{l_p}b\right)}, \\
            A_2(B) &:= \prod_{b \in \Psi_B}\prod_{q \mid p^{l_p}b} q^{\varphi\left(p^{l_p}b\right)/(q-1)}.
	\end{align*}
Here and in the following, $q$ will always denote a prime number. In particular, the above product in the definition of $A_2(B)$ is understood to run over all prime divisors of $p^{l_p}b$.
\end{definition}
Since $n$ is a multiple of $P_B$, both $A_1(B)^n$ and $A_2(B)^n$ are integers; also, $\theta_{\max}n$ is an integer.

\bigskip

For a rational function $R(t)=P(t)/Q(t)\in \mathbb{Q}(t)$, where $P(t),Q(t)$ are polynomials in $t$, we define the degree of $R(t)$ by $\deg R:= \deg P - \deg Q$. By Definition \ref{definition_Rn(t)}, we have
\[\deg R_n = -(n+1)s+n(|\mathcal{Z}_B|-1)+\theta_{\max}n-1. \]
Since $s > |\mathcal{Z}_B|$ and $\theta_{\max} < 1$, we have
\[ \deg R_n \leqslant -2  \]
for every $n \in I$.

By \eqref{size_of_Z_B}, we have $|\mathcal{Z}_B| = (a_p/2 +o(1))B^2$ as $B \to +\infty$. We will eventually take $B = B(s) = \widetilde{c}_p\sqrt{s/\log s}$ for some constant $\widetilde{c}_p >0$ depending only on $p$ and $\varepsilon$. Thus, the assumption $s > |\mathcal{Z}_B| > 2$ in Definition \ref{definition_Rn(t)} will be satisfied for any sufficiently large $s$.

Recall that the $p$-adic logarithm $\log_p$ is defined on $\{ t \in \mathbb{Q}_p ~\mid~ |t-1|_p < 1 \}$ by
\[ \log_p(1+x) = \sum_{j=1}^{\infty} (-1)^{j-1}\frac{x^j}{j}, \quad |x|_p < 1. \]
The function $f(t) = \log_p\langle t \rangle$ is defined on $\mathbb{Q}_p^{\times}$ and $f'(t) = 1/t$ for every $t \in \mathbb{Q}_p^{\times}$.

\begin{definition}
We denote the partial fraction decomposition of $R_n(t)$ by
	\begin{equation}\label{definition_r_ik}
		R_n(t) =: \sum_{i=1}^{s}\sum_{k=0}^{n} \frac{r_{i,k}}{(t+k)^i},
	\end{equation} 
	where the coefficients $r_{i,k} \in \mathbb{Q}$ are uniquely determined by $R_n(t)$. For every $n \in I$, we define the function $\widetilde{R}_n(t)$ by
	\begin{equation}\label{definition_widetildeR_n}
	 \widetilde{R}_n(t) := \sum_{k=0}^{n} r_{1,k}\log_p\langle t+k \rangle + \sum_{i=2}^{s}\sum_{k=0}^{n} \frac{r_{i,k}}{(1-i)(t+k)^{i-1}}. 
	 \end{equation}
\end{definition}

\begin{lemma}\label{rho_1_is_zero}
Define $\rho_1 := \sum_{k=0}^{n} r_{1,k}$. Then we have $\rho_1 = 0$. 
\end{lemma}
\begin{proof}
By \eqref{definition_r_ik} and $\deg R_n \leqslant -2$, we have $\rho_1 = \lim_{t \to \infty} tR_n(t) = 0$.
\end{proof}

\begin{lemma}\label{lemma_primitive}
	The function $\widetilde{R}_n(t)$ is a primitive function of $R_n(t)$ on $\mathbb{Q}_p \setminus \{0,-1,-2,\ldots,-n\}$; that is,
	\[ \widetilde{R}_n^{\prime}(t) = R_n(t), \quad\text{for any~} t \in \mathbb{Q}_p \setminus \{0,-1,-2,\ldots,-n\}. \]
	Moreover, for any $\theta \in \mathcal{Z}_B$, the function
	\[ \mathbb{Z}_p\to \mathbb{Q}_p, \quad t\mapsto \widetilde{R}_n(t +\theta),  \]
	is overconvergent, i.e. $\widetilde{R}_n(t +\theta)\in C^\dagger(\mathbb{Z}_p,\mathbb{Q}_p)$.
\end{lemma}

\begin{proof}
For $t \in \mathbb{Q}_p \setminus \{0,-1,-2,\ldots,-n\}$, we have
\[ \widetilde{R}_n^{\prime}(t) = \sum_{k=0}^{n} \frac{r_{1,k}}{t+k} + \sum_{i=2}^{s}\sum_{k=0}^{n} \frac{r_{i,k}}{(t+k)^i} = R_n(t). \]

For any $\theta \in \mathcal{Z}_B$ we have $|\theta|_p \geqslant p^{l_p} \geqslant q_p$.  Then for any $k \in \{0,1,\ldots,n\}$ and any $t \in \mathbb{Z}_p$, we have $|k+\theta|_p = |\theta|_p \geqslant q_p$ and  $\langle t+k+\theta \rangle = \langle 1 + (k+\theta)^{-1}t\rangle\langle k+\theta\rangle = (1+(k+\theta)^{-1}t)\langle k+\theta \rangle$. So 
\begin{align}
	&\widetilde{R}_n(t+\theta) = \sum_{k=0}^{n} r_{1,k}\log_p\langle k+\theta \rangle \notag\\
	&+ \sum_{k=0}^{n} r_{1,k}\log_p(1+(k+\theta)^{-1}t) + \sum_{i=2}^{s}\sum_{k=0}^{n} \frac{r_{i,k}}{(1-i)(k+\theta)^{i-1}}(1+(k+\theta)^{-1}t)^{1-i}, \quad t \in \mathbb{Z}_p .\label{Rn(t+theta)}
\end{align} 
Clearly, each summand on the right-hand side of \eqref{Rn(t+theta)} belongs to $C^\dagger(\mathbb{Z}_p,\mathbb{Q}_p)$. Therefore, $\widetilde{R}_n(t+\theta)$ is overconvergent. The proof of Lemma \ref{lemma_primitive} is complete.
\end{proof}

\section{Linear forms} \label{sec:LinearForms}

In this section, we define for each $\theta\in \mathcal{Z}_B$ a linear form in $p$-adic Hurwitz zeta functions as a Volkenborn integral over $-\widetilde{R}_n(t+\theta)$. By taking suitable sums of these linear forms, we obtain many linear forms in $p$-adic zeta values with related coefficients. These linear forms will serve as the key input for the elimination technique.

For any $\theta \in \mathcal{Z}_B$, by \eqref{definition_widetildeR_n} we have
\[ \widetilde{R}_n(t+\theta) = \sum_{k=0}^{n} r_{1,k}\log_p\langle t+k+\theta \rangle + \sum_{i=2}^{s}\sum_{k=0}^{n} \frac{r_{i,k}}{(1-i)(t+k+\theta)^{i-1}}.  \]
Note that each summand on the right-hand side above is Volkenborn integrable.

\begin{definition}
	For any $\theta \in \mathcal{Z}_B$, we define
	\[ S_{\theta} := -\int_{\mathbb{Z}_p} \widetilde{R}_n(t+\theta) \mathrm{d}t. \]
\end{definition}

It turns out that $S_{\theta}$ is a linear form in $1$ and $p$-adic Hurwitz zeta values.

\begin{lemma}\label{lemma_linear_form_S_theta}
For any $\theta \in \mathcal{Z}_B$, we have
\[ S_{\theta} = \rho_{0,\theta} + \sum_{i=2}^s \rho_i \cdot \omega(\theta)^{1-i}\zeta_p(i,\theta),  \]
where the coefficients
\begin{equation}\label{definition_rho_i}
\rho_{i}:=\sum_{k=0}^{n} r_{i,k}, \quad (2 \leqslant i \leqslant s)
\end{equation}
do not depend on $\theta$, and 
\begin{equation}\label{definition_rho_0theta}
\rho_{0,\theta}:=-\sum_{i=1}^{s}\sum_{k=1}^n\sum_{\nu=0}^{k-1} \frac{r_{i,k}}{(\nu+\theta)^{i}}.
\end{equation}
\end{lemma}

\begin{proof}
We have 
\begin{equation}\label{S_theta}
S_{\theta} = -\sum_{k=0}^{n} r_{1,k}\int_{\mathbb{Z}_p} \log_p\langle t+k+\theta \rangle \mathrm{d}t - \sum_{i=2}^{s}\sum_{k=0}^{n} r_{i,k}\int_{\mathbb{Z}_p} \frac{\mathrm{d}t}{(1-i)(t+k+\theta)^{i-1}}.
\end{equation}
By Lemma \ref{lem:Volkenborn_translation}, for any $i\in\{2,3,\ldots,s\}$ and $k \in \{0,1,\ldots,n\}$ we have
\begin{equation}\label{int_log_p(t+k+theta)}
\int_{\mathbb{Z}_p} \log_p\langle t+k+\theta \rangle \mathrm{d}t = \int_{\mathbb{Z}_p} \log_p\langle t+\theta \rangle \mathrm{d}t + \sum_{\nu=0}^{k-1} \frac{1}{\nu +\theta}, 
\end{equation}
and
\begin{align}
\int_{\mathbb{Z}_p}\frac{\mathrm{d}t}{(1-i)(t+k+\theta)^{i-1}} &= \int_{\mathbb{Z}_p}\frac{\mathrm{d}t}{(1-i)(t+\theta)^{i-1}} + \sum_{\nu=0}^{k-1} \frac{1}{(\nu+\theta)^{i}} \notag\\
&= -\omega(\theta)^{1-i}\zeta_p(i,\theta) + \sum_{\nu=0}^{k-1} \frac{1}{(\nu+\theta)^{i}}. \label{int_(t+k+theta)^(1-i)}
\end{align}
(When $k=0$, the empty sum $\sum_{\nu=0}^{k-1}$ is understood as $0$.) Substituting \eqref{int_log_p(t+k+theta)} and \eqref{int_(t+k+theta)^(1-i)} into \eqref{S_theta}, we obtain
\[ S_{\theta} = \rho_{0,\theta} + \sum_{i=2}^s \rho_i \cdot \omega(\theta)^{1-i}\zeta_p(i,\theta) - \rho_1 \int_{\mathbb{Z}_p} \log_p\langle t+\theta \rangle \mathrm{d}t. \]
Since $\rho_1 = 0$ by Lemma \ref{rho_1_is_zero}, the proof of Lemma \ref{lemma_linear_form_S_theta} is complete.
\end{proof}

Next, we combine $S_{\theta}$ $(\theta \in \mathcal{Z}_B)$ further to construct linear forms in $1$ and $p$-adic zeta values. We first notice a simple property of the set $\mathcal{Z}_B$.

\begin{lemma}\label{lemma_a_simple_property_of_Z_B}
For any integers $j,b$ such that $b \in \Psi_B$, $1 \leqslant j \leqslant p^{l_p}b$ and $\gcd(j,p)=1$, we have 
\[  \frac{j}{p^{l_p}b} \in \mathcal{Z}_B. \]
\end{lemma}

\begin{proof}
Let $j' =j/\gcd(j,b)$ and $b'=b/\gcd(j,b)$. Since $\gcd(j,p)=1$, we have $\gcd(j',p^{l_p}b')=1$. Since $b' \mid b$ and $b \in \Psi_B$, we have $\varphi(p^{l_p}b') \mid \varphi(p^{l_p}b)$ and hence $\varphi(p^{l_p}b') \leqslant  \varphi(p^{l_p}b) \leqslant B$. Therefore, $b' \in \Psi_B$ and 
\[ \frac{j}{p^{l_p}b} = \frac{j'}{p^{l_p}b'} \in \mathcal{Z}_B \]
by the definitions of $\Psi_B$ and $\mathcal{Z}_B$.
\end{proof}

Now we define $S_b$ for any $b \in \Psi_B$. It turns out that $S_b$ is a linear form in $1$ and $p$-adic odd zeta values.

\begin{definition}\label{definition_linear_form_S_b}
For any $b \in \Psi_B$, we define
\[ S_b := \sum_{1 \leqslant j \leqslant p^{l_p}b \atop \gcd(j,p)=1} S_{j/p^{l_p}b}. \]
\end{definition}

\begin{lemma}\label{lemma_linear_form_S_b}
	For any $b \in \Psi_B$, we have
	\[ S_b = \rho_{0,b} + \sum_{3 \leqslant i \leqslant s \atop i \text{~odd}} \rho_i \cdot (p^{l_p}b)^{i} \zeta_p(i), \]
	where the coefficients $\rho_i$ are those in \eqref{definition_rho_i} and 
	\begin{equation}\label{definition_rho_0b}
		\rho_{0,b} := \sum_{1 \leqslant j \leqslant p^{l_p}b \atop \gcd(j,p)=1} \rho_{0,j/p^{l_p}b}.
	\end{equation}
\end{lemma}

\begin{proof}
For any integer $j$ such that $1 \leqslant j \leqslant p^{l_p}b$ and $\gcd(j,p)=1$, by Lemma \ref{lemma_a_simple_property_of_Z_B} we have $j/p^{l_p}b \in \mathcal{Z}_B$. Therefore, by Lemma \ref{lemma_linear_form_S_theta} and Lemma \ref{lem:padicHurwitz_zeta}, we have
\begin{align*}
 S_b &= \sum_{1 \leqslant j \leqslant p^{l_p}b \atop \gcd(j,p)=1} \rho_{0,j/p^{l_p}b} + \sum_{i=2}^{s} \rho_i \sum_{1 \leqslant j \leqslant p^{l_p}b \atop \gcd(j,p)=1} \omega\left(\frac{j}{p^{l_p}b}\right)^{1-i}\zeta_p\left(i,\frac{j}{p^{l_p}b}\right) \\
 &= \rho_{0,b} + \sum_{i=2}^{s} \rho_i \cdot (p^{l_p}b)^{i} \zeta_p(i).
\end{align*} 
Since $\zeta_p(2k)=0$ for every positive integer $k$, the proof of Lemma \ref{lemma_linear_form_S_b} is complete.
\end{proof}

\section{Arithmetic properties}\label{sec:Arithmetic}
The goal of this section is to study the arithmetic properties of the coefficients of the linear forms constructed in section \ref{sec:LinearForms}. In particular, we will bound the denominators of these coefficients. Furthermore, we will also study the $\ell(n)$-adic valuations of these coefficients. This will be important later on, when we apply our irrationality criterion Lemma \ref{lem:irrationalityCrit}.

\bigskip

As usual, we denote by $d_n := \operatorname{LCM}\{1,2,\ldots,n\}$ the least common multiple of the smallest $n$ positive integers. For any non-negative integer $\lambda$, we define the differential operator
\[D_{\lambda}:=\frac{1}{\lambda !} \left(\frac{\mathrm{d}}{\mathrm{d}t}\right)^{\lambda}.\] 

\begin{lemma}\label{lemma_G(t)}
Let $n$ be any non-negative integer and let $G(t)=n!/(t)_{n+1}$. Then we have
\[ d_n^{\lambda} D_{\lambda}\left( G(t)(t+k) \right)\big|_{t=-k} \in \mathbb{Z} \]
for any integer $k \in \{0,1,\ldots,n\}$ and any integer $\lambda \geqslant 0$.
\end{lemma}

\begin{proof}
See \cite[Lemma 16]{Zud2004}.
\end{proof}

\begin{lemma}\label{lemma_F(t)}
Let $n$ be any non-negative integer. Let $a,b$ be any integers with $b>0$. Consider the polynomial
\[ F(t) = b^{n}\left(\prod_{q \mid b} q^{\left\lfloor n/(q-1) \right\rfloor} \right)\cdot \frac{\left(t+ \frac{a}{b} \right)_{n}}{n!}.  \]
\begin{enumerate}
\item[(1)] For any integer $k \in \mathbb{Z}$ and any integer $\lambda \geqslant 0$ we have
\[ d_n^{\lambda} D_{\lambda} (F(t)) \big|_{t=-k} \in \mathbb{Z}. \]
\item[(2)] For any integer $k' \in \{0,1,\ldots,n-1\}$, any integer $k \in \mathbb{Z}$ and any integer $\lambda \geqslant 0$ we have
\[ d_n^{\lambda+1} D_{\lambda} \left(F(t)\cdot\frac{1}{t+\frac{a}{b}+k'}\right) \Bigg|_{t=-k} \in \mathbb{Z}. \]
\end{enumerate}

\end{lemma}

\begin{proof}
For $(1)$, see \cite[Prop. 3.2]{LY2020}.

Now we prove $(2)$. Let 
\begin{align*}
	F_1(t) &= b^{k'}\left(\prod_{q \mid b} q^{\left\lfloor k'/(q-1) \right\rfloor}\right) \cdot \frac{\left(t+ \frac{a}{b} \right)_{k'}}{(k')!}, \\
	F_2(t) &= b^{n-1-k'}\left(\prod_{q \mid b} q^{\left\lfloor (n-1-k')/(q-1) \right\rfloor} \right)\cdot \frac{\left(t+ \frac{a}{b}+k'+1 \right)_{n-1-k'}}{(n-1-k')!}.
\end{align*}
Then we have
\[ F(t)\cdot\frac{1}{t+\frac{a}{b}+k'} = A\cdot\frac{(k')!(n-1-k')!}{n!}\cdot F_1(t)F_2(t),  \]
where
\[ A = b\prod_{q \mid b} q^{\left\lfloor n/(q-1) \right\rfloor-\left\lfloor k'/(q-1) \right\rfloor-\left\lfloor (n-1-k')/(q-1) \right\rfloor} \in \mathbb{Z}. \]
Applying the Leibniz rule, we obtain that
\begin{align*}
&d_n^{\lambda+1} D_{\lambda} \left(F(t)\cdot\frac{1}{t+\frac{a}{b}+k'}\right) \Bigg|_{t=-k} \\
=& A \cdot d_n\frac{(k')!(n-1-k')!}{n!} \cdot \sum_{\lambda_1,\lambda_2 \geqslant 0 \atop \lambda_1+\lambda_2=\lambda} d_n^{\lambda_1}D_{\lambda_1}(F_1(t))\cdot d_n^{\lambda_2}D_{\lambda_2}(F_2(t)) \big|_{t=-k}. 
\end{align*} 
By $(1)$, we have
\[ d_n^{\lambda_1}D_{\lambda_1}(F_1(t))\big|_{t=-k} \in \mathbb{Z},\quad d_n^{\lambda_2}D_{\lambda_2}(F_2(t)) \big|_{t=-k} \in \mathbb{Z}.   \]
Kummer's theorem on the $\ell$-adic valuation of binomial coefficients implies for any prime $\ell$
\[
    v_{\ell}\left( \binom{n}{n-k'}\right)\leqslant \left\lfloor \frac{\log n}{\log \ell} \right\rfloor - v_{\ell}(n-k')=v_{\ell}(d_n)-v_{\ell}(n-k'),
\]
and we deduce
\[ d_n\frac{(k')!(n-1-k')!}{n!}=\frac{d_n}{n-k'} \frac{1}{\binom{n}{n-k'}} \in \mathbb{Z}. \]
So we have 
\[ d_n^{\lambda+1} D_{\lambda} \left(F(t)\cdot\frac{1}{t+\frac{a}{b}+k'}\right) \Bigg|_{t=-k} \in \mathbb{Z}. \]
\end{proof}

\begin{lemma}\label{lemma_r_ik}
For any $i \in \{1,2,\ldots,s\}$ and any $k \in \{0,1,\ldots,n\}$, we have
\[ d_n^{s-i} r_{i,k} \in \mathbb{Z}. \]
\end{lemma}

\begin{proof}
By \eqref{definition_r_ik}, we have
\[r_{i,k} = D_{s-i}(R_n(t)(t+k)^{s})\big|_{t=-k}.\]
Let
\begin{align}
G(t) &= \frac{n!}{(t)_{n+1}}, \label{def_G(t)}\\
F_{\theta}(t) &= (p^{l_p}b)^{n}\left(\prod_{q \mid p^{l_p}b} q^{n/(q-1)}\right) \cdot \frac{(t+\theta)_{n}}{n!} \quad\text{for any~} \theta = \frac{a}{p^{l_p}b} \in \mathcal{Z}_B \setminus \{\theta_{\max}\}, \label{def_F_theta(t)}\\
F_{\theta_{\max}}(t) &= (p^{l_p}b_{\max})^{\theta_{\max}n-1}\left(\prod_{q \mid p^{l_p}b_{\max}} q^{\left\lfloor{(\theta_{\max}n-1)/(q-1)}\right\rfloor}\right)\cdot \frac{(t+\theta_{\max})_{\theta_{\max}n-1}}{(\theta_{\max}n-1)!}. \label{def_F_theta_max(t)}
\end{align}
Then it is straightforward to check that for any $n \in I$ we have
\begin{equation}\label{Rn(t)_building_blocks}
R_n(t) = A \cdot G(t)^s \prod_{\theta \in \mathcal{Z}_B} F_{\theta}(t), 
\end{equation}
where
\[ A = (p^{l_p}b_{\max})^{n-\theta_{\max}n+1}\prod_{q \mid p^{l_p}b_{\max}} q^{n/(q-1)-\left\lfloor{(\theta_{\max}n-1)/(q-1)}\right\rfloor} \in \mathbb{Z}. \]
Applying the Leibniz rule, we obtain
\begin{equation}\label{eqn_r_ik}
d_n^{s-i}r_{i,k} = A\sum_{\boldsymbol{\lambda}}\prod_{j=1}^{s}d_n^{\lambda_j}D_{\lambda_j}(G(t)(t+k))\prod_{\theta \in \mathcal{Z}_B}d_n^{\lambda_{\theta}}D_{\lambda_{\theta}}(F_{\theta}(t)) \big|_{t=-k},
\end{equation}
where the sum is taken over all families of non-negative integers 
\[\boldsymbol{\lambda} = ((\lambda_{j})_{1 \leqslant j \leqslant s},(\lambda_{\theta})_{\theta \in \mathcal{Z}_B})\] 
such that
\[ \sum_{j=1}^{s} \lambda_j + \sum_{\theta \in \mathcal{Z}_B} \lambda_{\theta} = s-i. \]
By Lemma \ref{lemma_G(t)} and $(1)$ of Lemma \ref{lemma_F(t)}, the value at $t=-k$ of each factor in the product in \eqref{eqn_r_ik} is an integer. Therefore, we have $d_n^{s-i}r_{i,k} \in \mathbb{Z}$ as desired. 
\end{proof}

\begin{lemma}\label{lemma_rho_i}
For any $i \in \{2,3,\ldots,s\}$, we have
\[ d_n^{s-i} \rho_{i} \in \mathbb{Z}. \]
\end{lemma}

\begin{proof}
It follows immediately from \eqref{definition_rho_i} and Lemma \ref{lemma_r_ik}.
\end{proof}

\begin{lemma}\label{lemma_rho_0theta_general}
For any $\theta_0 \in \mathcal{Z}_{B} \setminus \{ \theta_{\min} \}$, we have
\[ d_{n}^{s} \rho_{0,\theta_0} \in \mathbb{Z}. \]
\end{lemma}

\begin{proof}
By \eqref{definition_rho_0theta}, we have
\[\rho_{0,\theta_0} = \sum_{k=1}^n\sum_{\nu=0}^{k-1} \left(-\sum_{i=1}^{s}\frac{r_{i,k}}{(\nu+\theta_0)^{i}} \right). \] 
It is sufficient to prove that, for any pair of integers $(\nu_0,k_0)$ such that $0 \leqslant \nu_0 < k_0 \leqslant n$, we have
\begin{equation}\label{Inner_sum_is_good}
d_n^{s} \cdot \left(-\sum_{i=1}^{s}\frac{r_{i,k_0}}{(\nu_0+\theta_0)^{i}} \right) \in \mathbb{Z}.
\end{equation}

In the following, we prove \eqref{Inner_sum_is_good}. By \eqref{definition_r_ik}, we have $r_{i,k_0} = D_{s-i}(R_n(t)(t+k_0)^s)\big|_{t=-k_0}$. On the other hand, we have
\[ -\frac{1}{(\nu_0+\theta_0)^{i}} = D_{i-1}\left(  \frac{1}{t+k_0-\nu_0-\theta_0} \right)\Bigg|_{t=-k_0}. \]
Therefore, we compute by the Leibniz rule that
\begin{align*}
-\sum_{i=1}^{s}\frac{r_{i,k_0}}{(\nu_0+\theta_0)^{i}} &= \sum_{i=1}^{s} D_{s-i}(R_n(t)(t+k_0)^s)\cdot D_{i-1}\left(  \frac{1}{t+k_0-\nu_0-\theta_0} \right)\Bigg|_{t=-k_0} \\
&= D_{s-1}\left( R_n(t)(t+k_0)^s \cdot  \frac{1}{t+k_0-\nu_0-\theta_0} \right)\Bigg|_{t=-k_0}.
\end{align*}
By \eqref{Rn(t)_building_blocks}, we have
\begin{align*}
&R_n(t)(t+k_0)^s \cdot  \frac{1}{t+k_0-\nu_0-\theta_0} \\ 
=& A\cdot \left( F_{1-\theta_0}(t)\cdot\frac{1}{t+1-\theta_0+(k_0-\nu_0-1)} \right) \cdot \left(G(t)(t+k_0)\right)^s\prod_{\theta \in \mathcal{Z}_B \setminus\{1-\theta_0\}} F_{\theta}(t),
\end{align*}
where $A$ is an integer and the functions $G(t)$ and $F_{\theta}(t)$ ($\theta \in \mathcal{Z}_B$) are defined by \eqref{def_G(t)}, \eqref{def_F_theta(t)} and \eqref{def_F_theta_max(t)}. Applying the Leibniz rule again, we have
\begin{align*}
&d_n^{s} \cdot \left(-\sum_{i=1}^{s}\frac{r_{i,k_0}}{(\nu_0+\theta_0)^{i}} \right) \\
=& A \sum_{\boldsymbol{\lambda}} d_n^{\lambda_{*}+1}D_{\lambda_{*}}\left( F_{1-\theta_0}(t)\cdot\frac{1}{t+1-\theta_0+(k_0-\nu_0-1)} \right) \\ &\qquad\qquad\times\prod_{j=1}^{s}d_n^{\lambda_j}D_{\lambda_j}(G(t)(t+k_0))\prod_{\theta \in \mathcal{Z}_B\setminus\{1-\theta_0\}}d_n^{\lambda_\theta}D_{\lambda_\theta}(F_{\theta}(t))\big|_{t=-k_0},
\end{align*}
where the sum is taken over all families of non-negative integers 
\[\boldsymbol{\lambda} = (\lambda_*, (\lambda_{j})_{1 \leqslant j \leqslant s},(\lambda_{\theta})_{\theta \in \mathcal{Z}_B\setminus\{1-\theta_0\}})\]
such that
\[ \lambda_* + \sum_{j=1}^{s} \lambda_j + \sum_{\theta \in \mathcal{Z}_B\setminus\{1-\theta_0\}} \lambda_{\theta} = s-1. \]
Since $\theta_0 \neq \theta_{\min}$, we have $1- \theta_0 \neq \theta_{\max}$ and $\deg F_{1-\theta_0}(t) = n$. Note that $k_0-\nu_0-1 \in [0,n-1]$. Therefore, by $(2)$ of Lemma \ref{lemma_F(t)}, we have
\[ d_n^{\lambda_{*}+1}D_{\lambda_{*}}\left( F_{1-\theta_0}(t)\cdot\frac{1}{t+1-\theta_0+(k_0-\nu_0-1)} \right)\Bigg|_{t=-k_0} \in \mathbb{Z}. \]
By Lemma \ref{lemma_G(t)} and $(1)$ of Lemma \ref{lemma_F(t)}, we also have
\[ d_n^{\lambda_j}D_{\lambda_j}(G(t)(t+k_0))\big|_{t=-k_0} \in \mathbb{Z}, \quad d_n^{\lambda_\theta}D_{\lambda_\theta}(F_{\theta}(t))\big|_{t=-k_0} \in \mathbb{Z}. \]
We conclude that \eqref{Inner_sum_is_good} is true; thus, the proof of Lemma \ref{lemma_rho_0theta_general} is complete.
\end{proof}

\bigskip

Recall that $s>|\mathcal{Z}_B|$ and $\ell(n) = n+1$ is a prime number when $n \in I$. 

\begin{lemma}\label{lemma_rho_0theta_min}
For any sufficiently large $n \in I$, we have $\ell(n)^sd_n^s \rho_{0,\theta_{\min}} \in \mathbb{Z}$, and
\[ v_{\ell(n)}(\rho_{0,\theta_{\min}}) \leqslant  -s + |\mathcal{Z}_B| - 1. \]
\end{lemma}
	
\begin{proof}
By \eqref{definition_rho_0theta}, we have
\begin{align}
\rho_{0,\theta_{\min}} &= \sum_{k=1}^n\sum_{\nu=0}^{k-1} \left(-\sum_{i=1}^{s}\frac{r_{i,k}}{(\nu+\theta_{\min})^{i}} \right) \notag\\
&= \sum_{(\nu,k) \in \Gamma} \left(-\sum_{i=1}^{s}\frac{r_{i,k}}{(\nu+\theta_{\min})^{i}} \right), \label{eqn1}
\end{align}
where $\Gamma$ is the set
\[ \Gamma = \left\{ (\nu,k)\in\mathbb{Z}^2 \mid 0 \leqslant \nu < k \leqslant n \right\}. \]
We split $\Gamma$ into three disjoint subsets: $\Gamma = \Gamma_1 \cup \Gamma_2 \cup \{ (\theta_{\min}n,n) \}$, where
\begin{align*}
	\Gamma_1 &= \left\{ (\nu,k) \in \Gamma \mid \nu < \theta_{\min}n \right\}, \\
	\Gamma_2 &= \left\{ (\nu,k) \in \Gamma \mid \nu > \theta_{\min}n \text{~or~} \nu = \theta_{\min}n \text{~and~} k < n \right\}.
\end{align*}
In the following, we study the denominator of 
\[ -\sum_{i=1}^{s}\frac{r_{i,k_0}}{(\nu_0+\theta_{\min})^{i}} \]
in three cases. Recall that $\theta_{\min}=\frac{1}{p^{l_p}b_{\max}}$.

Case $(1)$: $(\nu_0,k_0) \in \Gamma_1$. In this case, we have $0<p^{l_p}b_{\max}\nu_0 + 1 \leqslant n$. So, $d_n \cdot (\nu_0 + \theta_{\min})^{-1} = d_n p^{l_p}b_{\max} \cdot (p^{l_p}b_{\max}\nu_0 + 1)^{-1} \in \mathbb{Z}$. We have $d_n^{s-i}r_{i,k_0} \in \mathbb{Z}$ by Lemma \ref{lemma_r_ik}. Therefore,
\begin{equation} \label{case_1}
d_n^s \cdot\left(-\sum_{i=1}^{s}\frac{r_{i,k_0}}{(\nu_0+\theta_{\min})^{i}}\right) = - \sum_{i=1}^{s} d_n^{s-i}r_{i,k_0} \cdot \left(d_n \cdot (\nu_0 + \theta_{\min})^{-1}\right)^i \in \mathbb{Z} \quad\text{for~} (\nu_0,k_0) \in \Gamma_1.
\end{equation}

Case $(2)$: $(\nu_0,k_0) \in \Gamma_2$. Similar to the proof of Lemma \ref{lemma_rho_0theta_general}, we have 
\begin{align*}
	&d_n^{s} \cdot \left(-\sum_{i=1}^{s}\frac{r_{i,k_0}}{(\nu_0+\theta_{\min})^{i}} \right) \\
	=& A \sum_{\boldsymbol{\lambda}} d_n^{\lambda_{*}+1}D_{\lambda_{*}}\left( F_{\theta_{\max}}(t)\cdot\frac{1}{t+\theta_{\max}+(k_0-\nu_0-1)} \right) \\ &\qquad\qquad\times\prod_{j=1}^{s}d_n^{\lambda_j}D_{\lambda_j}(G(t)(t+k_0))\prod_{\theta \in \mathcal{Z}_B\setminus\{\theta_{\max}\}}d_n^{\lambda_\theta}D_{\lambda_\theta}(F_{\theta}(t))\big|_{t=-k_0},
\end{align*}
where $A$ is an integer, the functions $G(t)$ and $F_{\theta}(t)$ ($\theta \in \mathcal{Z}_B$) are defined by \eqref{def_G(t)}, \eqref{def_F_theta(t)} and \eqref{def_F_theta_max(t)}, and the sum $\sum_{\boldsymbol{\lambda}}$ is taken over all families of non-negative integers 
\[\boldsymbol{\lambda} = (\lambda_*, (\lambda_{j})_{1 \leqslant j \leqslant s},(\lambda_{\theta})_{\theta \in \mathcal{Z}_B\setminus\{\theta_{\max}\}})\]
such that
\[ \lambda_* + \sum_{j=1}^{s} \lambda_j + \sum_{\theta \in \mathcal{Z}_B\setminus\{\theta_{\max}\}} \lambda_{\theta} = s-1. \]
Note that $\deg F_{\theta_{\max}}(t) = \theta_{\max}n-1$. And $(\nu_0,k_0) \in \Gamma_2$ implies that $k_0-\nu_0-1 \in [0,\theta_{\max}n-2]$. By $(2)$ of Lemma \ref{lemma_F(t)} (with $n$ replaced by $\theta_{\max}n-1$), we have
\[ \left. d_n^{\lambda_{*}+1}D_{\lambda_{*}}\left( F_{\theta_{\max}}(t)\cdot\frac{1}{t+\theta_{\max}+(k_0-\nu_0-1)} \right)\right|_{t=-k_0} \in \mathbb{Z}. \]
By Lemma \ref{lemma_G(t)} and $(1)$ of Lemma \ref{lemma_F(t)}, we also have
\[ d_n^{\lambda_j}D_{\lambda_j}(G(t)(t+k_0))\big|_{t=-k_0} \in \mathbb{Z}, \quad d_n^{\lambda_\theta}D_{\lambda_\theta}(F_{\theta}(t))\big|_{t=-k_0} \in \mathbb{Z}. \]
Thus, 
\begin{equation}\label{case_2}
d_n^s \cdot\left(-\sum_{i=1}^{s}\frac{r_{i,k_0}}{(\nu_0+\theta_{\min})^{i}}\right) \in \mathbb{Z} \quad\text{for~} (\nu_0,k_0) \in \Gamma_2.
\end{equation}

Case $(3)$: $(\nu_0,k_0) = (\theta_{\min}n,n)$. In this case, $\nu_0+\theta_{\min} = \ell(n)/p^{l_p}b_{\max}$. So, with the help of Lemma \ref{lemma_r_ik}, we have
\begin{multline}\label{case_3}
\ell(n)^sd_n^s \cdot \left(-\sum_{i=1}^{s}\frac{r_{i,k_0}}{(\nu_0+\theta_{\min})^{i}}\right)\\ = -\sum_{i=1}^{s} d_n^{s}r_{i,k_0} \cdot \ell(n)^{s-i}\cdot(p^{l_p}b_{\max})^{i}\in \mathbb{Z} \text{~for~} (\nu_0,k_0)= (\theta_{\min}n,n). 
\end{multline}
Substituting \eqref{case_1}, \eqref{case_2} and \eqref{case_3} into \eqref{eqn1}, we obtain that
\[ \ell(n)^sd_n^s \rho_{0,\theta_{\min}} \in \mathbb{Z}. \]
Now we study the $\ell(n)$-adic order for case $(3)$. By \eqref{definition_r_ik} we have
\begin{equation}\label{eqn2} -\sum_{i=1}^{s}\frac{r_{i,k_0}}{(\nu_0+\theta_{\min})^{i}} = -R_n(-\theta_{\max}n+\theta_{\min}) + \sum_{i=1}^{s}\sum_{k=0}^{n-1} \frac{r_{i,k}}{(-\theta_{\max}n+\theta_{\min}+k)^i}. 
\end{equation}
For any $0 \leqslant k \leqslant n-1$, we have
\[ -\theta_{\max}n+\theta_{\min}+k = \frac{1-p^{l_p}b_{\max}}{p^{l_p}b_{\max}}\ell(n) + k +1, \]
so $v_{\ell(n)}(-\theta_{\max}n+\theta_{\min}+k) = 0$ if $n$ is sufficiently large ($n > p^{l_p}b_{\max}$). By Lemma \ref{lemma_r_ik}, we have $v_{\ell(n)}(r_{i,k}) \geqslant 0$. Therefore,
\begin{equation}\label{eqn3}
v_{\ell(n)}\left( \sum_{i=1}^{s}\sum_{k=0}^{n-1} \frac{r_{i,k}}{(-\theta_{\max}n+\theta_{\min}+k)^i} \right) \geqslant 0. 
\end{equation}
On the other hand, we claim that
\begin{equation}\label{eqn4}
v_{\ell(n)}\left( R_n(-\theta_{\max}n+\theta_{\min}) \right) \leqslant -s + |\mathcal{Z}_B| - 1 < 0.
\end{equation}
In fact, for sufficiently large $n \in I$, clearly we have
\begin{align*}
			v_{\ell(n)}\left( A_1(B)^{n}A_2(B)^{n}\frac{n!^{s-|\mathcal{Z}_B|+1}}{(\theta_{\max}n-1)!} \right) = 0.
\end{align*}
Since 
\[ (t+\theta_{\max})_{\theta_{\max}n-1} \big|_{t = -\theta_{\max}n+\theta_{\min}} = (-1)^{\theta_{\max}n-1}(\theta_{\max}n-1)!, \]
we have
\[ v_{\ell(n)}\left( (t+\theta_{\max})_{\theta_{\max}n-1} \big|_{t = -\theta_{\max}n+\theta_{\min}}\right) = 0. \]
For any $\theta \in \mathcal{Z}_{B} \setminus \{ \theta_{\max}\}$, 
\[ (t+\theta)_{n} \big|_{t = -\theta_{\max}n+\theta_{\min}} = \prod_{j=0}^{n-1} (t+j+\theta) \big|_{t = -\theta_{\max}n+\theta_{\min}} \]
is a product of $n$ factors. Among these $n$ factors, at most one factor has positive $\ell(n)$-adic valuation, and no factor has $\ell(n)$-adic valuation $\geqslant 2$ (since $n$ is sufficiently large). So we have
\[ v_{\ell(n)}\left( (t+\theta)_{n} \big|_{t = -\theta_{\max}n+\theta_{\min}} \right) \leqslant 1 \quad\text{for any~} \theta \in \mathcal{Z}_{B} \setminus \{ \theta_{\max}\}. \]
At last, 
\begin{multline*}
(t)_{n+1} \big|_{t = -\theta_{\max}n+\theta_{\min}} = \prod_{j=1}^{n+1} \left( \frac{1-p^{l_p}b_{\max}}{p^{l_p}b_{\max}}\ell(n)+j \right)\\=\prod_{j=1}^{n+1} \left( \frac{\ell(n)}{p^{l_p}b_{\max}}-\ell(n)+j \right)= \frac{\ell(n)}{p^{l_p}b_{\max}} \cdot \prod_{j=1}^{n} \left( \frac{\ell(n)}{p^{l_p}b_{\max}} - j \right).
\end{multline*}
So 
\[ v_{\ell(n)}\left( (t)_{n+1} \big|_{t = -\theta_{\max}n+\theta_{\min}} \right) = 1.\]
In conclusion, we have
\[ v_{\ell(n)}\left( R_n(t) \big|_{t = -\theta_{\max}n+\theta_{\min}}\right) \leqslant -s + |\mathcal{Z}_B| - 1. \]
Thus \eqref{eqn4} is true. Combining \eqref{eqn2},\eqref{eqn3} and \eqref{eqn4}, we have
\[ v_{\ell(n)} \left( -\sum_{i=1}^{s}\frac{r_{i,k_0}}{(\nu_0+\theta_{\min})^{i}} \right) \leqslant -s + |\mathcal{Z}_B| - 1 \quad\text{for~} (\nu_0,k_0) = (\theta_{\min}n,n). \]
By \eqref{case_1} and \eqref{case_2}, we have
\[ v_{\ell(n)} \left( -\sum_{i=1}^{s}\frac{r_{i,k_0}}{(\nu_0+\theta_{\min})^{i}} \right) \geqslant 0 \quad\text{for~} (\nu_0,k_0) \in \Gamma_1 \cup \Gamma_2.  \] 
Therefore, by \eqref{eqn1} we obtain that
\[ v_{\ell(n)}(\rho_{0,\theta_{\min}}) \leqslant  -s + |\mathcal{Z}_B| - 1. \]
The proof of Lemma \ref{lemma_rho_0theta_min} is complete.
\end{proof}

\begin{lemma}\label{lemma_integrality_rho_0_b}
For any $b \in \Psi_{B} \setminus \{b_{\max}\}$, we have 
\[ d_n^{s}\rho_{0,b} \in \mathbb{Z}. \]
For $b= b_{\max}$, we have
\[ \ell(n)^sd_n^s\rho_{0,b_{\max}} \in \mathbb{Z} \]
and 
\[ v_{\ell(n)}(\rho_{0,b_{\max}}) \leqslant -s+|\mathcal{Z}_B| -1. \]	
\end{lemma}

\begin{proof}
It follows immediately from \eqref{definition_rho_0b}, Lemma \ref{lemma_rho_0theta_general} and Lemma \ref{lemma_rho_0theta_min}.
\end{proof}

\section{$p$-adic norm}\label{sec:p-adic}
In this section, we will bound the $p$-adic norm of the linear forms constructed in section \ref{sec:LinearForms}. Recall that we have assumed $ s > |\mathcal{Z}_B|$. 

\begin{lemma}\label{lemma_u_k}
For any $\theta \in \mathcal{Z}_B$, we have $R_n(t+\theta) \in C^\dagger(\mathbb{Z}_p,\mathbb{Q}_p)$. Moreover, if we write
\[ R_n(t+\theta) = \sum_{k=0}^{\infty} u_k t^{k}, \]
then $u_k \in \mathbb{Q}$ and 
\[ v_p(u_k) \geqslant \left( l_p + \frac{1}{p-1} \right)sn - \frac{\log(n+1)}{\log p}(s-|\mathcal{Z}_B|) +l_p(s+k) \]
for any integer $k \geqslant 0$.
\end{lemma}

\begin{proof}
Fix $\theta = a_0/p^{l_p}b_0 \in \mathcal{Z}_B$, where $b_0 \in \Psi_B$, $1 \leqslant a_0 \leqslant p^{l_p}b_0$ and $\gcd(a_0,p^{l_p}b_0) = 1$. Substituting $t+\theta$ for $t$ in Definition \ref{definition_Rn(t)}, we have
\begin{align*}
	R_n(t+\theta) &= A_1(B)^n A_2(B)^n \frac{n!^{s-|\mathcal{Z}_B|+1}}{(\theta_{\max}n-1)!} \\
	&\qquad\times \prod_{b \in \Psi_B} \prod_{1 \leqslant a \leqslant p^{l_p}b \atop \gcd(a,p^{l_p}b)=1}\prod_{k=0}^{n-1-\delta(a,b)} \left( t+\frac{a}{p^{l_p}b} + \frac{a_0}{p^{l_p}b_{0}} + k \right) \\
	&\qquad\times \prod_{k=0}^{n} \left( t+ \frac{a_0}{p^{l_p}b_{0}} + k \right)^{-s},
\end{align*} 
where
\[ \delta(a,b) = \begin{cases} \theta_{\min}n+1, &\text{if~} b = b_{\max} \text{~and~} a=p^{l_p}b_{\max} - 1, \\  0, &\text{otherwise.}  \end{cases} \]
By multiplying or dividing by a suitable integer for each factor of the product, we have
\begin{align}
	R_n(t+\theta) &= (p^{l_p}b_{\max})^{\theta_{\min}n+1} A_2(B)^n \frac{n!^{s-|\mathcal{Z}_B|+1}}{(\theta_{\max}n-1)!} \cdot p^{(n+1)sl_p} b_0^{(n+1)s - n|\mathcal{Z}_B| + \theta_{\min}n+1} \label{constant_term}\\
	&\qquad\times \prod_{b \in \Psi_B} \prod_{1 \leqslant a \leqslant p^{l_p}b \atop \gcd(a,p^{l_p}b)=1}\prod_{k=0}^{n-1-\delta(a,b)} \left( p^{l_p}bb_0t+ab_0 + a_0b + p^{l_p}bb_0k \right) \label{product_term_1}\\
	&\qquad\times \prod_{k=0}^{n} \left( p^{l_p}b_0t+ a_0 + p^{l_p}b_0k \right)^{-s}. \label{product_term_2}
\end{align}
Clearly, the product in the line \eqref{product_term_1} belongs to $\mathbb{Z}[p^{l_p}t]$. Noticing that $a_0$ is a $p$-adic unit, the product in the line \eqref{product_term_2} belongs to $\mathbb{Z}_p\llbracket p^{l_p}t \rrbracket \cap \mathbb{Q}\llbracket p^{l_p}t \rrbracket$. Therefore, $R_n(t+\theta)$ is the product of the constant in the line \eqref{constant_term} and a power series in $\mathbb{Z}_p\llbracket p^{l_p}t \rrbracket \cap \mathbb{Q}\llbracket p^{l_p}t \rrbracket$. The radius of convergence of this power series is at least $p^{l_p} > 1$, so $R_n(t+\theta)$ is overconvergent. Moreover, if we write
\[ R_n(t+\theta) = \sum_{k=0}^{\infty} u_k t^{k}, \]
then $u_k \in \mathbb{Q}$ and 
\begin{align}
	v_p(u_k) &\geqslant v_p\left( \text{the constant in the line \eqref{constant_term}} \right) + l_pk \notag\\
	&\geqslant v_p\left( A_2(B)^{n} n!^{s-|\mathcal{Z}_B|} p^{(n+1)sl_p} \right) + l_pk. \label{v_p(u_k)}
\end{align}
Obviously, we have
\begin{equation}\label{v_p(A_2)}
	v_p(A_2(B)^n) = \frac{n}{p-1} \sum_{b \in \Psi_B} \varphi(p^{l_p}b) = \frac{n|\mathcal{Z}_B|}{p-1} 
\end{equation}
and
\begin{equation}\label{v_p(n!)}
	v_p(n!) \geqslant \frac{n}{p-1} - \frac{\log(n+1)}{\log p}.
\end{equation}
Substituting \eqref{v_p(A_2)} and \eqref{v_p(n!)} into \eqref{v_p(u_k)}, we obtain the desired inequality for $v_p(u_k)$.
\end{proof}

\begin{lemma}\label{lemma_h_k}
	For any $t \in \mathbb{Q}_p$ such that $|t|_p > 1$, we have
	\[ R_n(t) = \sum_{k=(n+1)s-n|\mathcal{Z}_B|+\theta_{\min}n+1}^{\infty} \frac{h_k}{t^k}, \]
	where $h_k \in \mathbb{Z}$ and
	\[ v_p(h_k) \geqslant \frac{sn}{p-1}-\frac{\log(n+1)}{\log p}(s-|\mathcal{Z}_B|) + l_p\max\{0,(n+1)s-k\} \]
	for any integer $k \geqslant (n+1)s-n|\mathcal{Z}_B|+\theta_{\min}n+1$. Moreover, for any $t \in \mathbb{Q}_p$ such that $|t|_p \geqslant q_p$, we have
	\[ \widetilde{R}_n(t) = \sum_{k=(n+1)s-n|\mathcal{Z}_B|+\theta_{\min}n+1}^{\infty} \frac{h_{k}}{(1-k)t^{k-1}}. \]
\end{lemma}

\begin{proof}
By Definition \ref{definition_Rn(t)}, we have
\begin{align*}
	R_n(t) &= A_1(B)^n A_2(B)^n \frac{n!^{s-|\mathcal{Z}_B|+1}}{(\theta_{\max}n-1)!} \\
	&\qquad\times \prod_{b \in \Psi_B} \prod_{1 \leqslant a \leqslant p^{l_p}b \atop \gcd(a,p^{l_p}b)=1}\prod_{k=0}^{n-1-\delta(a,b)} \left( t+\frac{a}{p^{l_p}b} + k \right) \\
	&\qquad\times \prod_{k=0}^{n} \left( t + k \right)^{-s},
\end{align*}
where
\[ \delta(a,b) = \begin{cases} \theta_{\min}n+1, &\text{if~} b = b_{\max} \text{~and~} a=p^{l_p}b_{\max} - 1, \\  0, &\text{otherwise.}  \end{cases} \]
Thus,
\begin{align}
R_n(t) &= (p^{l_p}b_{\max})^{\theta_{\min}n+1} A_2(B)^n \frac{n!^{s-|\mathcal{Z}_B|+1}}{(\theta_{\max}n-1)!} \cdot \frac{1}{t^{(n+1)s-n|\mathcal{Z}_B|+\theta_{\min}n+1}} \label{C}\\
&\qquad\times \prod_{b \in \Psi_B} \prod_{1 \leqslant a \leqslant p^{l_p}b \atop \gcd(a,p^{l_p}b)=1}\prod_{k=0}^{n-1-\delta(a,b)} \left( p^{l_p}b+ \frac{a + p^{l_p}bk}{t} \right) \label{P(t)}\\
&\qquad\times \prod_{k=0}^{n} \left( 1 + \frac{k}{t} \right)^{-s}. \label{Q(t)}
\end{align}
Let us write
\[ R_n(t) = \frac{C}{t^{(n+1)s-k_0}} \cdot P(t) \cdot Q(t), \]
where $C$ is the constant factor in the line \eqref{C}, $P(t)$ is the product in the line \eqref{P(t)}, $Q(t)$ is the product in the line \eqref{Q(t)}, and $k_0 = n|\mathcal{Z}_B|-\theta_{\min}n-1$. 

Clearly, $P(t)$ is a polynomial in $t^{-1}$ of degree $k_0$. Moreover, if we write
\[ P(t) = a_0 + a_1 t^{-1} + \cdots + a_{k_{0}}t^{-k_0}, \]
then we have $a_j \in \mathbb{Z}$ and $v_p(a_j) \geqslant l_p(k_0-j)$ for any $j=0,1,\ldots,k_0$. It is obvious that $Q(t) \in \mathbb{Z}\llbracket t^{-1} \rrbracket$. Hence, $P(t)Q(t) \in \mathbb{Z}\llbracket t^{-1} \rrbracket$. Moreover, if we write 
\[ P(t)Q(t) = b_0 + b_1 t^{-1} + b_2t^{-2} + \cdots, \]
then $v_p(b_j) \geqslant l_p\max\{ 0, k_0 - j\}$ for any integer $j \geqslant 0$. In particular, we get for any $k\geqslant (n+1)s-k_0$ the inequality
\[
    v_p(b_{k-(n+1)s+k_0})\geqslant l_p \cdot \max\{ 0, (n+1)s-k \}.
\]
Therefore, we have 
\begin{equation}\label{h_k_power_series}
R_n(t) = \sum_{k=(n+1)s-k_0}^{\infty} h_k t^{-k}, 
\end{equation}
where $h_k = C b_{k-(n+1)s+k_0} \in \mathbb{Z}$ and
\begin{align}
	v_p(h_k) &= v_p(C) + v_p(b_{k-(n+1)s+k_0}) \notag\\
	&\geqslant v_p\left( A_2(B)^n n!^{s-|\mathcal{Z}_B|} \right) + l_p\max\{ 0, (n+1)s-k\} \label{v_p(h_k)}
\end{align}
for any integer $k \geqslant (n+1)s-k_0$. Substituting \eqref{v_p(A_2)} and \eqref{v_p(n!)} into \eqref{v_p(h_k)}, we obtain the desired inequality for $v_p(h_k)$.

Finally, since $\langle t+k \rangle = \langle 1+kt^{-1} \rangle\langle t \rangle = (1+kt^{-1})\langle t \rangle$ for any $t \in \mathbb{Q}_p$ such that $|t|_p \geqslant q_p$ and any $k \in \{0,1,\ldots,n\}$, we have
\begin{align*}
	\widetilde{R}_n(t) &= \rho_1 \log_p\langle t \rangle + \sum_{k=0}^{n} r_{1,k}\log_p\left(1+\frac{k}{t}\right) + \sum_{i=2}^{s}\sum_{k=0}^{n} \frac{r_{i,k}}{(1-i)}\frac{1}{t^{i-1}}\left(1+\frac{k}{t}\right)^{1-i} \\
 &= \sum_{k=0}^{n} r_{1,k}\log_p\left(1+\frac{k}{t}\right) + \sum_{i=2}^{s}\sum_{k=0}^{n} \frac{r_{i,k}}{(1-i)}\frac{1}{t^{i-1}}\left(1+\frac{k}{t}\right)^{1-i}
\end{align*}
by \eqref{definition_widetildeR_n} and Lemma \ref{rho_1_is_zero}. So $\widetilde{R}_n(t) \in t^{-1}\mathbb{Q}\llbracket t^{-1} \rrbracket$. Suppose that $\widetilde{R}_n(t) = \sum_{k=1}^{\infty} \widetilde{h}_k t^{-k}$. Taking derivative and comparing it with \eqref{h_k_power_series}, we obtain that
\[ \widetilde{R}_n(t) = \sum_{k=(n+1)s-n|\mathcal{Z}_B|+\theta_{\min}n+1}^{\infty} \frac{h_{k}}{(1-k)t^{k-1}}, \qquad |t|_p \geqslant q_p. \]
The proof of Lemma \ref{lemma_h_k} is complete. 
\end{proof}

\begin{lemma}\label{lemma_p_adic_norm_S_theta}
For any $\theta \in \mathcal{Z}_B$, we have
\[ \limsup_{n \to \infty} |S_{\theta}|_p^{1/n} \leqslant p^{-(l_p+1/(p-1))s}. \]
\end{lemma}
	
\begin{proof}
By Lemma \ref{lemma_primitive}, we have $\widetilde{R}_n(t+\theta) \in C^\dagger(\mathbb{Z}_p,\mathbb{Q}_p)$ and $\widetilde{R}_n^{\prime}(t+\theta) = R_n(t+\theta)$. Applying Lemma \ref{lem:comparison_Volkenborn_Bernoulli} to the function $f(t) = \widetilde{R}_n(t+\theta) - \widetilde{R}_n(\theta)$, we obtain
\begin{equation}\label{S_theta=L_1+widetildeRn(theta)}
S_{\theta} = -\mathcal{L}_1(R_n(t+\theta)) - \widetilde{R}_n(\theta).
\end{equation}

By Lemma \ref{lemma_u_k}, we have 
\[
 R_n(t+\theta) = \sum_{k=0}^{\infty} u_k t^{k} \in C^\dagger(\mathbb{Z}_p,\mathbb{Q}_p), 
\]
and
\[ v_p(u_k) \geqslant \left( l_p + \frac{1}{p-1} \right)sn - \frac{\log(n+1)}{\log p}(s-|\mathcal{Z}_B|) +l_p(s+k) \]
for any integer $k \geqslant 0$. Therefore,
\[ \mathcal{L}_1(R_n(t+\theta)) = \sum_{k=0}^{\infty} u_k\frac{B_{k+1}}{k+1} \]
and hence
\begin{align}
	\left| \mathcal{L}_1(R_n(t+\theta)) \right|_p &\leqslant \max_{k \geqslant 0} \left| u_k\frac{B_{k+1}}{k+1}  \right|_p \leqslant \max_{k \geqslant 0}  p(k+1)|u_k|_p  \notag\\
	&\leqslant \max_{k \geqslant 0} p^{-(l_p+1/(p-1))sn}(n+1)^{s-|\mathcal{Z}_B|}\frac{p(k+1)}{p^{l_p(s+k)}} \notag\\
	&= p^{-(l_p+1/(p-1))sn}(n+1)^{s-|\mathcal{Z}_B|} p^{1-l_ps}. \label{p_adic_norm_L_1_Rn(t)}
\end{align}

On the other hand, since $|\theta|_p \geqslant p^{l_p} \geqslant q_p$, by Lemma \ref{lemma_h_k}, we have
\[ \widetilde{R}_n(\theta) = \sum_{k=(n+1)s-n|\mathcal{Z}_B|+\theta_{\min}n+1}^{\infty} \frac{h_{k}}{(1-k)\theta^{k-1}}, \]
where $h_k \in \mathbb{Z}$ and
\[ v_p(h_k) \geqslant \frac{sn}{p-1}-\frac{\log(n+1)}{\log p}(s-|\mathcal{Z}_B|) + l_p\max\{0,(n+1)s-k\} \]
for any integer $k \geqslant (n+1)s-n|\mathcal{Z}_B|+\theta_{\min}n+1$. Since $v_p(\theta) \leqslant -l_p$, we have
\begin{align*}
 &v_p\left( \frac{h_{k}}{(1-k)\theta^{k-1}} \right) \\
 &\geqslant v_p(h_k) + l_p(k-1) - \frac{\log(k-1)}{\log p} \\
 &\geqslant \frac{sn}{p-1}-\frac{\log(n+1)}{\log p}(s-|\mathcal{Z}_B|) + l_p\max\{k-1,(n+1)s-1\} - \frac{\log(k-1)}{\log p} \\
 &\geqslant  \frac{sn}{p-1}-\frac{\log(n+1)}{\log p}(s-|\mathcal{Z}_B|) + l_p((n+1)s-1) - \frac{\log((n+1)s-1)}{\log p}
\end{align*}
for any integer $k\geqslant (n+1)s-n|\mathcal{Z}_B|+\theta_{\min}n+1$. Therefore, we have
\begin{equation}\label{p_adic_norm_widetildeRn(theta)}
	|\widetilde{R}_n(\theta)|_p \leqslant p^{-(l_p+1/(p-1))sn}(n+1)^{s-|\mathcal{Z}_B|}((n+1)s-1) p^{-l_p(s-1)}.
\end{equation}
Substituting \eqref{p_adic_norm_L_1_Rn(t)} and \eqref{p_adic_norm_widetildeRn(theta)} into \eqref{S_theta=L_1+widetildeRn(theta)}, we complete the proof of Lemma \ref{lemma_p_adic_norm_S_theta}.
\end{proof}

\section{Archimedean properties}\label{sec:Archimedean}

In the following, we will estimate the Archimedean growth of the coefficients of the linear forms in $1$ and $p$-adic zeta values as $n\to \infty$. Recall our assumption in Definition \ref{definition_Rn(t)} that $s > |\mathcal{Z}_B| > 2$.

\begin{lemma}\label{lemma_r_ik_estimate}
For any $i \in \{1,2,\ldots,s\}$ and any $k \in \{0,1,\ldots,n\}$, we have
\[  |r_{i,k}| \leqslant A_1(B)^nA_2(B)^n 2^{(s-|\mathcal{Z}_B|+2)n} \cdot (2p^{l_p}b_{\max})^{3s}n^{2s}. \]
\end{lemma}

\begin{proof}
By \eqref{definition_r_ik} and Cauchy's integral formula, we have
\[ r_{i,k} = \frac{1}{2\pi \sqrt{-1}} \int_{|z+k|=\frac{1}{2}\theta_{\min}} (z+k)^{i-1}R_n(z) \mathrm{d}z, \]
and hence
\begin{align}
|r_{i,k}| \leqslant &A_1(B)^nA_2(B)^n \frac{n!^{s-|\mathcal{Z}_B|+1}}{(\theta_{\max}n-1)!} \notag\\ 
&\times \sup_{|z+k|=\frac{1}{2}\theta_{\min}}\left| \frac{(z+\theta_{\max})_{\theta_{\max}n-1}\prod_{\theta\in\mathcal{Z}_B\setminus\{\theta_{\max}\}}(z+\theta)_{n}}{(z)_{n+1}^s} \right|. \label{pre_estimate_r_ik}
\end{align}
In the following, we estimate each of the terms in \eqref{pre_estimate_r_ik}.

In this paragraph, the complex number $z$ we consider always lies on the circle $|z+k| = \frac{1}{2}\theta_{\min}$. First, we have
\[ |(z)_{n+1}| \geqslant \prod_{\iota = 0}^{n} \left|  |\iota-k|-\frac{1}{2}\theta_{\min} \right|. \]
If $|\iota-k|>1$, then we have $\left|  |\iota-k|-\theta_{\min}/2 \right| \geqslant |\iota-k| - 1$; otherwise, we have $\left|  |\iota-k|-\theta_{\min}/2 \right| \geqslant \theta_{\min}/2$. We obtain the estimate
\begin{equation}\label{r_ik_estimate_part_1}
	\sup_{|z+k|=\frac{1}{2}\theta_{\min}} \frac{1}{|(z)_{n+1}|} \leqslant (2p^{l_p}b_{\max})^3 \cdot \frac{n^2}{k!(n-k)!}.
\end{equation}
Next, for any $\theta \in \mathcal{Z}_B\setminus\{\theta_{\max}\}$, we have 
\[|(z+\theta)_{n}| \leqslant \prod_{\iota=0}^{n-1} \left( |\iota-k+\theta|+\frac{1}{2}\theta_{\min} \right). \]
If $\iota \geqslant k$, then we have $|\iota-k+\theta|+\theta_{\min}/2 \leqslant \iota - k + 1$; otherwise, we have $|\iota-k+\theta|+\theta_{\min}/2 \leqslant k-\iota$. We obtain the estimate
\begin{equation}\label{r_ik_estimate_part_2}
	\sup_{|z+k|=\frac{1}{2}\theta_{\min}} |(z+\theta)_{n}| \leqslant  k!(n-k)!.
\end{equation}
Finally, we consider $(z+\theta_{\max})_{\theta_{\max}n-1}$. We have
\[ |(z+\theta_{\max})_{\theta_{\max}n-1}| \leqslant \prod_{\iota=0}^{\theta_{\max}n-2} \left( |\iota-k+\theta_{\max}| + \frac{1}{2}\theta_{\min} \right). \]
If $\iota \geqslant k$, then we have $|\iota-k+\theta_{\max}|+\theta_{\min}/2 \leqslant \iota - k + 1$; otherwise, we have $|\iota-k+\theta_{\max}|+\theta_{\min}/2 \leqslant k-\iota$. We obtain the estimate
\[ \sup_{|z+k|=\frac{1}{2}\theta_{\min}}|(z+\theta_{\max})_{\theta_{\max}n-1}| \leqslant 
\begin{cases} 
	k!(\theta_{\max}n-1-k)! &\text{~if~} 0 \leqslant k \leqslant \theta_{\max}n-1, \\
	\frac{k!}{(k-\theta_{\max}n+1)!} &\text{~if~} \theta_{\max}n \leqslant k \leqslant n.
\end{cases} \]
We always have
\begin{equation}\label{r_ik_estimate_part_3}
	\sup_{|z+k|=\frac{1}{2}\theta_{\min}}|(z+\theta_{\max})_{\theta_{\max}n-1}| \leqslant \frac{n!}{(\theta_{\min}n+1)!}.
\end{equation}

Now, substituting \eqref{r_ik_estimate_part_1}, \eqref{r_ik_estimate_part_2} and \eqref{r_ik_estimate_part_3} in \eqref{pre_estimate_r_ik}, we have
\begin{align*}
	|r_{i,k}| &\leqslant A_1(B)^nA_2(B)^n  \binom{n}{k}^{s-|\mathcal{Z}_B|+1} \binom{n}{\theta_{\min}n+1} \cdot (2p^{l_p}b_{\max})^{3s}n^{2s} \\
	&\leqslant A_1(B)^nA_2(B)^n 2^{(s-|\mathcal{Z}_B|+2)n} \cdot (2p^{l_p}b_{\max})^{3s}n^{2s}.
\end{align*}
The proof of Lemma \ref{lemma_r_ik_estimate} is complete.
\end{proof}

\begin{lemma}\label{lemma_rho_i_estimate}
For any $i \in \{2,3,\ldots,s\}$, we have the estimates
\[ \limsup_{n \to \infty} |\rho_i|^{1/n}  \leqslant A_1(B)A_2(B)2^s, \]
and, for any $\theta \in \mathcal{Z}_B$, we have
\[ \limsup_{n \to \infty} |\rho_{0,\theta}|^{1/n}  \leqslant A_1(B)A_2(B)2^s. \]
\end{lemma}

\begin{proof}
By \eqref{definition_rho_i} and \eqref{definition_rho_0theta}, we have
\[ |\rho_{i}| \leqslant \sum_{k=0}^{n} |r_{i,k}|, \]
and
\begin{align*}
 |\rho_{0,\theta}| &\leqslant \sum_{i=1}^{s}\sum_{k=1}^{n}\sum_{\nu=0}^{k-1} \frac{|r_{i,k}|}{(\nu+\theta)^{i}} \\
 & \leqslant  sn^2 \cdot \frac{1}{\theta_{\min}^s} \cdot \max_{i,k}|r_{i,k}| \\
 & = sn^2 \cdot (p^{l_p}b_{\max})^s \cdot \max_{i,k}|r_{i,k}|.
\end{align*}
Now, Lemma \ref{lemma_rho_i_estimate} follows immediately from Lemma \ref{lemma_r_ik_estimate} since $|\mathcal{Z}_B| > 2$.
\end{proof}

The previous lemma allows us to estimate the growth of the coefficients $\rho_i$ and $\rho_{0,\theta}$ of our linear forms as $n\to \infty$ in terms of $A_1(B)$, $A_2(B)$ and $s$. The next lemma discusses the dependence of $A_1(B)$ and $A_2(B)$ on $B$ as $B\to + \infty$. 

\begin{lemma}\label{lemma_estimate_A_1_A_2}
	As $B \to +\infty$, we have
	\[ A_1(B) = \exp\left( \left(\frac{a_p}{2}+o(1)\right)B^2\log B \right), \]
 where the constant $a_p$ is defined by \eqref{def_a_p}.
	On the other hand, for any $B$ larger than some constant depending at most on $p$, we have
	\[ A_2(B) \leqslant \exp\left( 10B^2(\log\log B)^2 \right). \]
\end{lemma}

\begin{proof}
	We start by 
	\[\log A_1(B) = \sum_{b \in \Psi_B} \varphi\left(p^{l_p}b\right) \log\left(p^{l_p}b\right). \] Define the function \[ f(x) := \#\left\{ b \in \mathbb{N} ~\mid~ \varphi(p^{l_p}b) \leqslant x \right\}, \quad x \in (0,+\infty). \]
	Firstly, we have
	\begin{align*}
		\log A_1(B) &\geqslant  \sum_{b \in \Psi_B} \varphi\left(p^{l_p}b\right) \log\varphi\left(p^{l_p}b\right) \\
		&= \int_{1^{-}}^{B} x\log x \mathrm{d} f(x) \\
		&= f(B)B\log B - \int_{1^{-}}^{B} f(x)(\log x + 1)\mathrm{d}x.
	\end{align*}
By Lemma \ref{lemma_inverse_totient}, we have $f(x) \sim a_p x$ as $x \to +\infty$. Therefore, we obtain that 
\begin{equation}\label{A_1_geq}
	\log A_1(B) \geqslant \left(\frac{a_p}{2}+o(1)\right)B^2\log B. 
\end{equation} 
On the other hand, it is well known (see \cite[Theorem 2.9]{MV2006}) that
\[ \varphi(m) \geqslant (e^{-\gamma}+o_{m \to +\infty}(1))\frac{m}{\log\log m}, \]
where $\gamma = 0.577\ldots$ is Euler's constant. For any $b \in \Psi_{B}$, since $\varphi\left(p^{l_p}b\right) \leqslant B$, we derive that
\begin{equation}\label{upper_bound_for_elements_in_Psi_B}
	b \leqslant \frac{e^{\gamma}+o(1)}{p^{l_p}}B\log\log B.
\end{equation}
Therefore, $\log\left(p^{l_p}b\right) \leqslant (1+o(1))\log B$ holds uniformly for $b \in \Psi_B$. We have 
\begin{align}
	\log A_1(B) &\leqslant (1+o(1))\log B \sum_{b \in \Psi_B} \varphi\left(p^{l_p}b\right) \notag\\
	&= (1+o(1))\log B \cdot |\mathcal{Z}_B| \notag\\
	&= \left(\frac{a_p}{2}+o(1)\right)B^2\log B \label{A_1_leq}
\end{align}
by \eqref{size_of_Z_B}. Combining \eqref{A_1_geq} and \eqref{A_1_leq}, we obtain the estimate for $A_1(B)$.

Now we consider $A_2(B)$. By \eqref{upper_bound_for_elements_in_Psi_B} and $e^\gamma = 1.78\ldots < 2$, when $B$ is larger than some absolute constant, we have $b \leqslant 2(B\log\log B)/p^{l_p}$ for every $b \in \Psi_B$.  Therefore, using our convention that $q$ denotes always a prime number, we get
\begin{align}
	\log A_2(B) &= \sum_{b \in \Psi_B} \varphi\left(p^{l_p}b\right) \sum_{q \mid p^{l_p}b} \frac{\log q}{q-1} \notag\\
	&= \frac{\log p}{p-1}\sum_{b \in \Psi_B} \varphi\left(p^{l_p}b\right) + \sum_{b \in \Psi_B} \varphi\left(p^{l_p}b\right) \sum_{q \mid b \atop q \neq p} \frac{\log q}{q-1} \notag\\
	&\leqslant \frac{\log p}{p-1}\sum_{b \in \Psi_B} \varphi\left(p^{l_p}b\right) + \sum_{b \leqslant 2(B\log\log B)/p^{l_p}} p^{l_p}b \sum_{q \mid b} \frac{\log q}{q-1} \notag\\
	&= \frac{\log p}{p-1}|\mathcal{Z}_B| + p^{l_p}\sum_{q \leqslant 2(B\log\log B)/p^{l_p}} \frac{\log q}{q-1} \sum_{b \leqslant 2(B\log\log B)/p^{l_p} \atop q \mid b} b. \label{A_2_leq}
\end{align}
Note that
\begin{align}
	&p^{l_p}\sum_{q \leqslant 2(B\log\log B)/p^{l_p}} \frac{\log q}{q-1} \sum_{b \leqslant 2(B\log\log B)/p^{l_p} \atop q \mid b} b \notag\\
	\leqslant& p^{l_p}\sum_{q \leqslant 2(B\log\log B)/p^{l_p}} \frac{\log q}{q-1} \cdot \frac{4B^2(\log\log B)^2}{p^{2l_p}q} \notag\\
	\leqslant& 4B^2(\log\log B)^2 \sum_{q} \frac{\log q}{q(q-1)} \notag\\
	\leqslant& 8B^2(\log\log B)^2. \label{A_2_main_part_leq}
\end{align}
The last inequality follows e.g.\ from the estimate
\[
\sum_{q} \frac{\log q}{q(q-1)}\leqslant \sum_{k=2}^\infty \frac{\log k}{k(k-1)}\leqslant \int_1^\infty \frac{\log x}{x(x-1)}\mathrm{d}x=\int_0^1 \frac{-\log x}{1-x}\mathrm{d}x =\zeta(2)< 2.
\]
Substituting \eqref{size_of_Z_B} and \eqref{A_2_main_part_leq} into \eqref{A_2_leq}, we obtain that
\[ \log A_2(B) \leqslant \frac{\log p}{p-1}\left(\frac{a_p}{2} + o(1)\right)B^2 + 8B^2(\log\log B)^2 \leqslant 10B^2(\log\log B)^2 \]
when $B$ is larger than some constant depending at most on $p$. The proof of Lemma \ref{lemma_estimate_A_1_A_2} is complete.
\end{proof}

\section{Elimination procedure and proof of the main theorem}\label{sec:MainThm}
In this section, we will apply the elimination technique to the linear forms constructed in section \ref{sec:LinearForms}. 

\begin{proposition}\label{prop:elimination}
    If $A_1(B)A_2(B)2^se^s<p^{\left(l_p+\frac{1}{p-1}\right)s}$ and $s> |\mathcal{Z}_B|>2$ then at least $|\Psi_B|$ of the numbers $\zeta_p(3),\zeta_p(5),\dots,\zeta_p(s)$ are irrational.
\end{proposition}
\begin{proof}
    Set $I_s:=\{3,5,\dots,s\}$. We argue by contradiction. Suppose there is a set $J\subseteq I_s$ of cardinality $|J|=|\Psi_B|-1$ such that $\zeta_p(i)\in \mathbb{Q}$ for all $i\in I_s\setminus J$. The generalized Vandermonde matrix
    \[
        ((p^{l_p}b)^j)_{b\in \Psi_B, j\in \{1\}\cup J}
    \]
    is invertible (see \cite[Lemma 4]{FSZ2019}), so we can find integers $w_b\in \mathbb{Z}$ for $b\in \Psi_B$ such that
    \begin{align*}
        \sum_{b\in \Psi_B} w_b (p^{l_p}b)^j&=0,\quad (j \in J)\\
        \sum_{b\in \Psi_B} w_b p^{l_p}b &= \det ((p^{l_p}b)^j)_{b\in \Psi_B, j\in \{1\}\cup J}.
    \end{align*}
    For $n\in I$ and $b\in \Psi_B$, we have defined in section \ref{sec:LinearForms} the linear forms in $1$ and $p$-adic zeta values
    \[ S_b = \rho_{0,b} + \sum_{3 \leqslant i \leqslant s \atop i \text{~odd}} \rho_i \cdot (p^{l_p}b)^{i} \zeta_p(i),  \quad (b \in \Psi_B)\]
    see Definition \ref{definition_linear_form_S_b} and Lemma \ref{lemma_linear_form_S_b}. Note that the coefficients $\rho_{0,b}$ and $\rho_i$ for $i\in I_s$ depend implicitly on $n$.
    For $n\in I$, we define the linear form
    \[
        L_n(X_0,(X_i)_{i\in I_s\setminus J}):=l_{0,n}X_0+\sum_{i\in I_s\setminus J}l_{i,n}X_i
    \]
    where
    \begin{align*}
         l_{0,n}:&=\ell(n)^sd_n^s\sum_{b\in \Psi_B}w_b \rho_{0,b},\\
         l_{i,n}:&=\ell(n)^sd_n^s\rho_{i}\sum_{b\in \Psi_B}w_b(p^{l_p}\cdot b)^i, \quad (i\in I_s\setminus J).
    \end{align*}
    By Lemma \ref{lemma_integrality_rho_0_b} and Lemma \ref{lemma_rho_i}, the coefficients $l_{0,n}$ and $l_{i,n}$ for $i\in I_s\setminus J$ are integral. Note that the linear forms $L_n$ are constructed in such a way that
    \[
        L_n(1,(\zeta_p(i))_{i\in I_s\setminus J})=\ell(n)^sd_n^s\sum_{b\in \Psi_B} w_b \cdot S_b.
    \]
    Hence, Lemma \ref{lemma_p_adic_norm_S_theta} allows us to estimate the $p$-adic norm of our linear forms:
    \begin{equation}\label{eq:Sn_padic}
        \limsup_{n} |L_n(1,(\zeta_p(i))_{i\in I_s\setminus J})|_p^{1/n}\leqslant p^{-\left(l_p+\frac{1}{p-1}\right)s}.
    \end{equation}
    On the other hand, we have for $i\in (I_s\setminus J)\cup \{0\}$ the following upper bound for the Archimedean growth of the coefficients
    \begin{equation}\label{eq:lni_Archimedean}
        \limsup_{n} |l_{i,n}|^{1/n}\leqslant A_1(B)A_2(B)2^se^s
    \end{equation}
    by Lemma \ref{lemma_rho_i_estimate} and $\lim_n |d_n|^{1/n}=e$. Note that the coefficient $w_{b_{\text{max}}}\neq 0$ of $\rho_{0,b_{\text{max}}}$ is non-zero. In fact, every $w_b$ is non-zero because it is itself a generalized Vandermonde matrix by Cramer's rule. By Lemma \ref{lemma_integrality_rho_0_b}, the term corresponding to $b_{\max}$ in 
    \[
    l_{0,n}=\ell(n)^sd_n^s\sum_{b\in \Psi_B}w_b \rho_{0,b}
    \]
    dominates the other terms $\ell(n)$-adically and thereby controls the $\ell(n)$-adic valuation of $l_{0,n}$. More precisely, we have by Lemma \ref{lemma_integrality_rho_0_b} for $n$ sufficiently large
    \[
        v_{\ell(n)}(l_{0,n})=v_{\ell(n)}(\ell(n)^sd_n^sw_{b_{\max}}\rho_{0,b_{\max}})=s+v_{\ell(n)}(\rho_{0,b_{\max}})\leqslant |\mathcal{Z}_B| -1.
    \]
    On the other hand, Lemma \ref{lemma_rho_i} allows us to estimate the $\ell(n)$-adic valuation of $l_{i,n}$ for $i\in I_s\setminus J$
    \[
        v_{\ell(n)}(l_{i,n})\geqslant s.
    \]
    Thus, using $s>|\mathcal{Z}_B|$, we get for $n \in I$ sufficiently large:
    \begin{equation}\label{eq:ln0_dominates}
        v_{\ell(n)}(l_{0,n})< v_{\ell(n)}(l_{i,n}) \text{ for all }i\in I_s\setminus J.
    \end{equation}
    We can now apply Lemma \ref{lem:irrationalityCrit}. Note that the condition \ref{lem:irrationalityCrit:i} of Lemma \ref{lem:irrationalityCrit} is satisfied by \eqref{eq:lni_Archimedean} and \eqref{eq:Sn_padic}, while condition \ref{lem:irrationalityCrit:ii}  follows from \eqref{eq:ln0_dominates} and condition \ref{lem:irrationalityCrit:iii} holds because $I$ is unbounded. We deduce that at least one of the remaining zeta values $\zeta_p(i)$ for $i\in I_s\setminus J$ is irrational. This contradicts the choice of $J$.
\end{proof}

\bigskip

\begin{proof}[Proof of the main theorem (Theorem \ref{thmA})]
For $\varepsilon>0$, let us define
\[
    B(s):=\frac{c_p-\varepsilon/2}{a_p}\sqrt{\frac{s}{\log s}}.
\]
Our goal is to check that the assumptions of Proposition \ref{prop:elimination} are satisfied for this choice of $B(s)$ when $s$ is sufficiently large. Lemma \ref{lemma_sizes_of_sets} shows
\[
    2 < |\mathcal{Z}_{B(s)}| = \left( \frac{a_p}{2} + o(1)\right)B(s)^2<s
\]
for $s$ sufficiently large. By Lemma \ref{lemma_estimate_A_1_A_2}, we have
\begin{equation}\label{eq:proof_main_thm_A1A2}
    A_1(B(s))\cdot A_2(B(s))=\exp\left( \left( \frac{1}{4}\frac{(c_p-\varepsilon/2)^2}{a_p} +o(1) \right)s \right)\quad \text{ as } s\to \infty.
\end{equation}
From the definition of $c_p$ in Theorem \ref{thmA} and the definition of $a_p$ in \eqref{def_a_p}, we get
\[
    c_p=\sqrt{4a_p \left(\left(l_p+\frac{1}{p-1}\right)\log p -1-\log 2\right)},
\]
and hence
\[
    \frac{1}{4}\frac{c_p^2}{a_p} +\log2+1=\left( l_p +\frac{1}{p-1}\right)\log p,
\]
so we deduce from \eqref{eq:proof_main_thm_A1A2} that
\[
    A_1(B(s))\cdot A_2(B(s))2^s e^s <p^{\left(l_p+\frac{1}{p-1}\right)s}.
\]
Now Proposition \ref{prop:elimination} implies that at least $|\Psi_{B(s)}|$ of the numbers $\zeta_p(3),\zeta_p(5),\dots,\zeta_p(s)$ are irrational for $s$ sufficiently large. On the other hand, Lemma \ref{lemma_sizes_of_sets} together with the definition of $B(s)$ shows that
\[
    |\Psi_{B(s)}|=(1+o(1)) (c_p-\varepsilon/2)\sqrt{\frac{s}{\log s}}.
\]
So we deduce that for sufficiently large odd integers $s$ there are at least $(c_p-\varepsilon)\sqrt{\frac{s}{\log s}}$ irrational numbers among $\zeta_p(3),\zeta_p(5),\dots,\zeta_p(s)$. This proves the main theorem.
\end{proof}

\bigskip



\vspace*{3mm}
\begin{flushright}
\begin{minipage}{148mm}\sc\footnotesize
L.\,L., Beijing International Center for Mathematical Research, Peking University, Beijing, China\\
{\it E--mail address}: {\tt lilaimath@gmail.com} \vspace*{3mm}
\end{minipage}
\end{flushright}

\begin{flushright}
\begin{minipage}{148mm}\sc\footnotesize
J.\,S., Department of Mathematics, University of Duisburg-Essen, Essen, Germany \\
{\it E--mail address}: {\tt johannes.sprang@uni-due.de } \vspace*{3mm}
\end{minipage}
\end{flushright}

\end{document}